\documentclass[a4paper,10pt]{article}
\usepackage[utf8x]{inputenc}
\usepackage[colorlinks=true, pdfstartview=FitV, linkcolor=purple, citecolor=purple]{hyperref}
\usepackage{amsmath}
\usepackage{amsthm}
\usepackage{lineno}
\usepackage{cleveref}
\usepackage{bbm}
\usepackage{pictex}
\usepackage{tikz}
\usetikzlibrary{shapes, positioning, arrows.meta,calc}
\usepackage[margin=0.8in]{geometry}
\usepackage{ bbold }
\usepackage{amssymb}
\usepackage{tabularx}
\usepackage{float}
\usepackage{import}
\usepackage{enumitem}
\usepackage[toc,page]{appendix}
\usepackage{listings}
\usepackage{xcolor} 
\usepackage{makeidx}
\makeindex

\newtheorem{theorem}{Theorem}[section]
\newtheorem{lemma}[theorem]{Lemma}

\newtheorem{corollary}[theorem]{Corollary}

\newcommand\lref[1]{Lemma~\ref{#1}}
\newcommand\floor[1]{\left\lfloor #1 \right\rfloor}

\title{A Renormalization Scheme for Semi-Regular Continued Fractions}

\author{Niels Langeveld and David Ralston}

\begin{document}
\maketitle
\section{Introduction}

In this article we will study a renormalization scheme with which we find all semi-regular continued fractions of a number in a natural way. In \Cref{sec:TslowTfast} we define two maps, $\hat{T}_{slow}$ and $\hat{T}_{fast}$: these maps are defined for $(x,y) \in [0,1]^2$, where $x$ is the number for which a semi-regular continued fraction representation is developed by $\hat{T}_{slow}$ according to the parameter $y$. Importantly, the set of all possible semi-regular continued fraction representations of $x$ are bijectively constructed as the parameter $y$ varies (\Cref{theorem:all CF expansions}), making $\hat{T}_{slow}$ a natural setup for discussing these representations of $x$. The map $\hat{T}_{fast}$ is a ``sped-up" version of the map $\hat{T}_{slow}$, and we show that $\hat{T}_{fast}$ is ergodic with respect to a probability measure which is mutually absolutely continuous with Lebesgue measure on $[0,1]^2$(\Cref{theorem - T_fast is ergodic}). In contrast, $\hat{T}_{slow}$ preserves no such measure, but does preserve an infinite, $\sigma$-finite measure mutually absolutely continuous with Lebesgue measure (\Cref{corollary - T_slow not ergodic}). 

In \Cref{section - symbolic coding} we show that the maps $\hat{T}_{slow}$ and $\hat{T}_{fast}$ applied to the point $(x,y)$ can generate a sequence of substitutions which generate a \textit{symbolic coding} of the orbit of $y$ with respect to the intervals $[0,1-x]$, $[1-x,1]$. These substitutions are shown to naturally relate to finding a sequence of those $n \in \mathbb{Z}^+$ such that $-nx \mod 1$ best approximates $y$ for all $-ix \mod 1$, $1 \leq i \leq n$ (\Cref{lemma - slow encoding lemma}. Ergodicity of $\hat{T}_{fast}$ then leads to statements of a generic growth rate for this sequence (\Cref{cor:approximating sequence growth rates}).

Finally, in \Cref{section - relation to continued fractions}, we highlight how our scheme can be used to generate semi-regular continued fractions, explicitly mentioning regular continued fractions \cite{DK2}, backward continued fractions \cite{R}, $\alpha$-continued fractions \cite{N}, a natural counterpart of $\alpha$-continued fractions \cite{KLMM}, and Lehner continued fractions \cite{L}. Ergodicity of $\hat{T}_{fast}$ and lack of ergodicity of $\hat{T}_{slow}$ also leads to a statement regarding the generic growth rate of denominators of convergents of $x$ across our parameterization of semi-regular continued fractions (\Cref{theorem:generic denominator growth}).

\section{The maps \texorpdfstring{$\hat{T}_{slow}$}{} and \texorpdfstring{$\hat{T}_{fast}$}{} }\label{sec:TslowTfast}
Let $X$ be the unit circle $\mathbb{R}/\mathbb{Z}$, but with all $y\in X$ having a ``left" and ``right" version $y^-$ and $y^+$, which satisfy $y^- < y^+$. Then we write $X=[0^+,1^-]$. We give $X$ the topology generated by all open sets of the form $(a^+,b^-)$ for $0 \leq a<b\leq 1$. For $x \in \mathbb{R}$, let $R_x$ denote rotation by $x$, i.e. $R_x(y^{\pm})=(y+x)^{\pm} \mod 1$. Observe that the sets $[0^+,(1-x)^-]$ and $[(1-x)^+,1^-]$ form a partition of $X$ into two disjoint compact sets. In studying properties of $R_x$, there is no meaningful distinction between $R_x^n(y^+)$ and $R_x^n(y^-)$ when $y \notin x\mathbb{Z} \mod 1$, so we identify such points with one another. We will therefore suppress the notation required by $X$, simply considering $y \in [0,1]$. Intervals always begin at the right-sided version of a number and end at the left-sided version if an endpoint is some $nx \mod 1$: e.g. $[0,1-x]$ and $[1-x,1]$ are disjoint under this convention.

Circle rotations are convenient to represent as the simplest nontrivial type of \emph{interval exchange transformation}, or \emph{IET}, a bijective orientation-preserving piecewise isometry of the interval $[0,1]$. Specifically, rotations can be represented as exchanges of two intervals (a $2$-IET). See \Cref{figure - rotation as 2-iet} for a presentation of $R_x$ in this way. Note that 
\[ 
R_x(y) = \begin{cases} y+x, & y \in [0,1-x],\\ y+x-1, & y \in [1-x,1].\end{cases}
\]

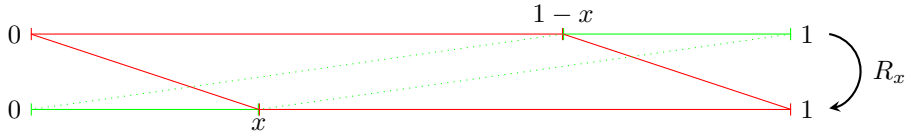
\begin{figure}[thb]
\center{\begin{tikzpicture}[xscale=1]
\draw[red, |-|] (0,1)--(7,1);
\draw[green, |-|] (7,1)--(10,1);
\draw[green, |-|] (0,0)--(3,0);
\draw[red, |-|] (3,0)--(10,0);
\draw[red] (0,1)--(3,0);
\draw[red] (7,1)--(10,0);
\draw[green, dotted] (7,1)--(0,0);
\draw[green, dotted] (10,1)--(3,0);
\node[left] at (0,1){$0$};
\node[left] at (0,0){$0$};
\node[right] at (10,1){$1$};
\node[right] at (10,0){$1$};
\node[above] at (7,1){$1-x$};
\node[below] at (3,0){$x$};

\draw[thick,-stealth]
        ($(10.5,1)$) arc
        [start angle=80,
        end angle=-80,
        radius=0.5] ;
\node[right] at (10.95,0.5){$R_x$};      
\end{tikzpicture}}
\caption{\label{figure - rotation as 2-iet}Rotation by $x$ presented as a $2$-IET.}
\end{figure}

Next, let $A=[1-x,1]$. For any $y \in A$, define the \emph{return time} of $y$ to be $N(y)$, where
\[
N(y) = \min\left\{ n \in \mathbb{Z}^+ : R_x^n(y) \in A \right\}.
\]
The \emph{first-return map on $A$} is given by $R_x^{N(y)}(y)$. It is well-known that the first-return map on $A$ has exactly two return times. Namely, for any $y\in A$ we either have $N(y)=a$ or $N(y)=a+1$ where $a$ is the largest positive integer so that $ax<1$, \textit{i.e.} $a=\lfloor 1/x \rfloor$, the first partial quotient of $x$.
For $y\in [-(a+1)x,1]$ (modulo one) we have $N(y)=a$ and for $y\in  [1-x,-(a+1)x]$ we have $N(y)=a+1$; see \Cref{figure - first return map construction}.
The first return map is therefore another $2$-IET given by 
\[ 
\hat{R}_{\hat{x}}(y) = 
\begin{cases} 
y+\hat{x}, & y \in [1-x,-(a+1)x \mod 1],\\
y+\hat{x} -1, & y \in [-(a+1)x \mod 1,1];
\end{cases}
\]
where $\hat{x}=(a+1)x -1$. This map is either a rotation by $(a+1)x-1$ or rotation by $1-ax$: the two choices are isomorphic, differing only by a choice of orientation. We adopt the convention that \emph{the shorter return time determines the orientation and rotation amount of the first-return map}. Our motivation for this choice is as follows:

\begin{figure}[bht]
\center{\begin{tikzpicture}[xscale=1]
\draw[|-|] (2,1)--(4,1);
\draw[|-|] (4,1)--(10,1);
\draw[dashed, |-|] (0,1)--(2,1);
\draw[|-|] (2,0)--(8,0);
\draw[|-|] (8,0)--(10,0);
\draw (4,1)--(2,0);
\draw (10,1)--(8,0);
\draw[dotted] (4,1)--(10,0);
\draw[dotted] (2,1)--(8,0);
\node[left] at (2,0){$1-x$};
\node[above] at (2,1){$1-x$};
\node[left] at (0,1){$0$};
\node[right] at (10,0){$1$};
\node[right] at (10,1){$1$};
\node[below] at (8,0){$ax$};
\node[above] at (4.5,1){$-(a+1)x \mod 1$};

\draw[thick,-stealth]
        ($(10.5,1)$) arc
        [start angle=80,
        end angle=-80,
        radius=0.5] ;
\node[right] at (10.95,0.5){$R_x^{N(y)}$};  

\end{tikzpicture}}
\caption{\label{figure - first return map construction}We construct the first-return map on the interval $[1-x,1]$. We obtain another $2$-IET (i.e. rotation) with two different return times; $a$ (solid lines) and $a+1$ (dotted lines).}
\end{figure}
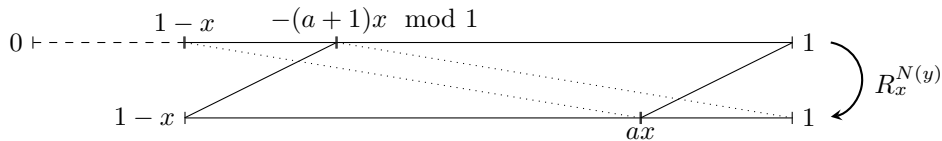

\begin{lemma}\label{lemma - first return}
The first-return map on $A=[1-x,1]$, after rescaling and subject to the convention that shorter return times dictate the choice of direction, is isomorphic to rotation by $T(x)$, where 
\[T(x) = \frac{1}{x} - \left\lfloor \frac{1}{x}\right\rfloor\]
is the Gauss map.

The map $\tau: [1-x,1] \mapsto [0,1]$ given by 
\begin{equation}\label{eqn - tau}\tau(y) = \frac{1-y}{x} = T(x) - \frac{y}{x} \mod 1\end{equation}
is an orientation-reversing isomorphism between the first return map on $[1-x,1]$ and rotation by $T(x)$ on $[0,1]$. 
\begin{proof}
We see in \Cref{figure - first return map construction} that within $A$, the shorter return time is subtraction of $(1-ax)$. After rescaling by $1/x$ to present this map as acting on an interval of length one, we have subtraction of $1/x - a = T(x)$. Since the direction of rotation is also determined by this shorter return time, we reverse orientation and conclude that the first-return map may be presented as $R_{T(x)}$. Referring again to \Cref{figure - first return map construction}, we see that we need $\tau$ to be a linear function which maps the interval $[1-x,1]$ to $[0,1]$ with reversed orientation: this determines $\tau$.
\end{proof}
\end{lemma}

Certainly \Cref{lemma - first return} is well-known regarding the rotation amount of the induced map; we present it because we will utilize the specific isomorphism $\tau$.

What happens if we construct the first-return map on $[0,1-x]$ instead? We may reconsider rotation by $x$ in one direction as rotation by $1-x$ in the other, so that by reversing orientation, now the interval $[0,1-x]$ is of length equal to the rotation amount. Applying \Cref{lemma - first return} then yields:

\begin{corollary}
The first-return map on $[0,1-x]$ is isomorphic to rotation by $T(1-x)$, where the isomorphism is given by 
\begin{equation} \label{eqn - alternate tau} \tau(y) = \frac{y}{1-x}. 
\end{equation}
\end{corollary}

Note that in contrast to \lref{lemma - first return}, in this scenario the isomorphism $\tau$ is orientation-preserving. We reversed orientation once to consider the original map $R_x$ to instead be $R_{1-x}$, and then the first-return map reversed orientation again by applying \Cref{lemma - first return}.

We present both first-return maps simultaneously in \Cref{figure - two first returns in one}: on $[1-x,1]$ we obtain rotation by $T(x)$ with orientation reversed, while on $[0,1-x]$ we obtain rotation by $T(1-x)$ with orientation preserved. Here $N^\prime(y)=\min\left\{ n \in \mathbb{Z}^+ : R_x^n(y) \in [0,1-x] \right\}.$ The lengths of the red and green arrows represent the rotation amount of this first-return map (not yet rescaled to be in an interval of length one), where the arrows begin at the point which is mapped to the origin by the isomorphism $\tau$ and point in the direction of the induced rotation.
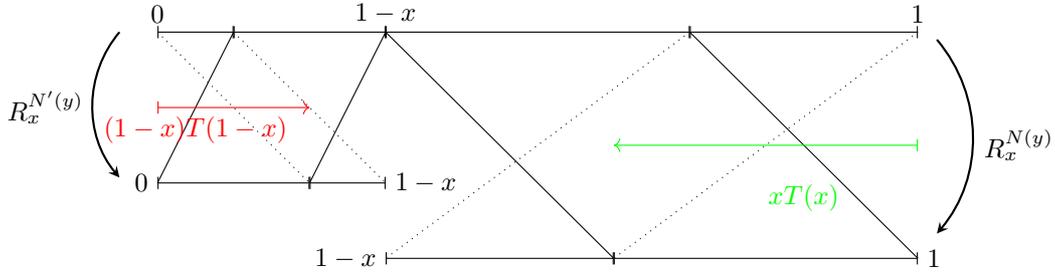
\begin{figure}[bht]
\center{\begin{tikzpicture}[xscale = 1]
\draw[|-|](0,3)--(1,3);
\draw[|-|](1,3)--(3,3);
\draw[|-|](3,3)--(7,3);
\draw[|-|](7,3)--(10,3);
\draw[|-|](0,1)--(2,1);
\draw[|-|](2,1)--(3,1);
\draw[|-|](3,0)--(6,0);
\draw[|-|](6,0)--(10,0);
\draw (1,3)--(0,1);
\draw (3,3)--(2,1);
\draw (3,3)--(6,0);
\draw (7,3)--(10,0);
\draw[dotted] (0,3)--(2,1);
\draw[dotted] (1,3)--(3,1);
\draw[dotted] (7,3)--(3,0);
\draw[dotted] (10,3)--(6,0);
\node [above] at (0,3){$0$};
\node [above] at (1,3){};
\node [above] at (3,3){$1-x$};
\node [above] at (7,3){};
\node [above] at (10,3){$1$};
\node [left] at (0,1){$0$};
\node [below] at (2,1){};
\node [right] at (3,1){$1-x$};
\node [left] at (3,0){$1-x$};
\node [below] at (6,0){};
\node [right] at (10,0){$1$};
\draw[red] [|->] (0,2)--(2,2);
\draw[green] [|->] (10,1.5)--(6,1.5);
\node [red,below] at (.5,2){$(1-x)T(1-x)$};
\node [green,above] at (8.5,.5){$xT(x)$};
\draw[thick,-stealth]
        ($(10.25,2.9)$) arc
        [start angle=40,
        end angle=-40,
        radius=2] ;
\node[right] at (10.75,1.5){$R_x^{N(y)}$};  
\draw[thick,-stealth]
        ($(-0.5,3)$) arc
        [start angle=140,
        end angle=220,
        radius=1.5] ;
\node[right] at (-2.1,2){$R_{x}^{N^\prime(y)}$};  
\end{tikzpicture}}
\caption{\label{figure - two first returns in one} The two different first-return maps, on $[0,1-x]$ on the left and on $[1-x,1]$ on the right.}
\end{figure}

We can now define our first new function: $\hat{T}_{slow}$ acts on the pair $(x,y)$, where $x$ represents a rotation amount and $y \in X$. Depending on whether $y \in [0,1-x]$ or $y \in [1-x,1]$ we take the first-return map on the appropriate interval (either $R_x^{N'(y)}$ or $R_x^{N(y)}$, respectively) along with the specified isomorphism $\tau$ to obtain a new rotation amount and point in the space $X$:
\begin{equation}\label{eqn - T_slow}
\hat{T}_{slow}(x,y) = 
\begin{cases} 
\left( T(1-x),\frac{y}{1-x} \right) &y \in [0,1-x], \\ 
\left( T(x), \frac{1-y}{x} \right)& y \in [1-x,1].
\end{cases}
\end{equation}

For iterations of this map we present a useful visual aid. Let $x_0=x$ be our rotation amount and $y_0=y \in [0,1]$, as before. Generate the sequences $\{x_n\}$, $\{y_n\}$, $\{\hat{a}_n\}$ as $(x_n,y_n)=\hat{T}_{slow}^n(x_0,y_0)$ and $\hat{a}_n$ the return time used to define the $n$-th first-return map. Then we may equivalently follow the algorithm given in \Cref{figure - graph presentation}.
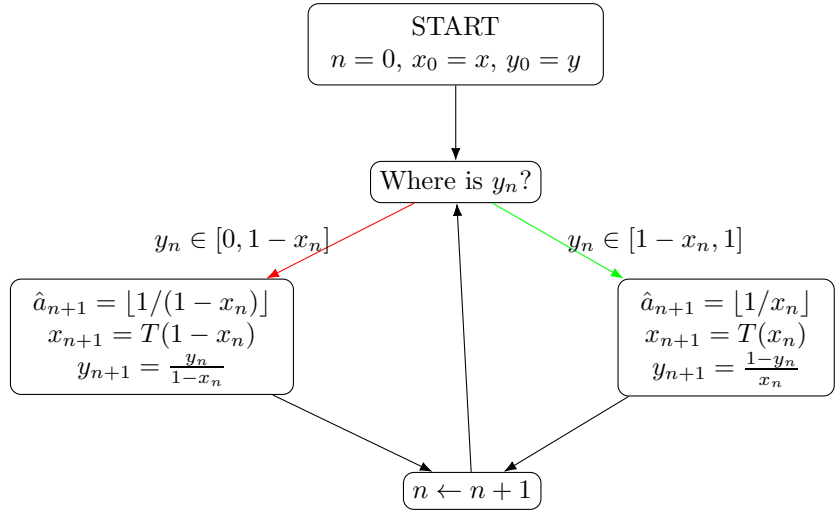
\begin{figure}[bth]
\center{\begin{tikzpicture}[xscale=1]
\node(start)[draw,rectangle, rounded corners]{$\begin{array}{c}\textrm{START}\\n=0, \, x_0=x, \, y_0 = y\end{array}$};
\node(set a)[below = of start, draw, rectangle, rounded corners]{Where is $y_n$?};
\node(regular)[below right = of set a, draw, rectangle, rounded corners]{$\begin{array}{c}\hat{a}_{n+1} = \lfloor 1/x_n \rfloor \\ x_{n+1} = T(x_n)\\y_{n+1} = \frac{1-y_n}{x_n} 
\end{array}$};
\node(reversing)[below left = of set a, draw, rectangle, rounded corners]{$\begin{array}{c}\hat{a}_{n+1} = \lfloor 1/(1-x_n) \rfloor \\ x_{n+1} = T(1-x_n)\\y_{n+1} = \frac{y_n}{1-x_n}
\end{array}$};
\node(increment)[below left = of regular,draw, rectangle, rounded corners]{$n \gets n+1$};
\draw[-Latex] (start) -- (set a);
\draw[-Latex,green] (set a) -- (regular) node[black,midway,right]{$y_n \in [1-x_n,1]$};
\draw[-Latex,red] (set a) -- (reversing) node[black,midway,left]{$y_n \in [0,1-x_n]$};
\draw[-Latex] (regular) -- (increment);
\draw[-Latex] (reversing) -- (increment);
\draw[-Latex] (increment) -- (set a);
\end{tikzpicture}}
\caption{\label{figure - graph presentation}
An informal graph algorithm showing the action of $\hat{T}_{slow}$. The choice of green versus red arrows aligns with the color scheme in \Cref{figure - two first returns in one}.}
\end{figure}

At first glance the presentation in \Cref{figure - graph presentation} does not add anything beyond the definition of $\hat{T}_{slow}$ in \Cref{eqn - T_slow}. The two colored arrows, however, provide a clear way to distinguish when $\hat{T}_{slow}$ acts on the first coordinate $x_n$ via the Gauss map (the green arrow) or a modified version thereof (the red arrow), and in both cases forces us to compute the shorter return times $\hat{a}$, which we will make use of later. The following lemma is a technical necessity, but also introduces a critical link between the map $\hat{T}_{slow}$ and well-known techniques in the study of continued fractions:

\begin{lemma}\label{lemma - fast returns defined}
For any initial choice of $(x_0,y_0)$, any execution of the algorithm in \Cref{figure - graph presentation} must have infinitely many edges which are either green, or red with $\hat{a}_n\geq 2$; equivalently, no $(x_0,y_0)$ enters into an infinite loop of following the red arrow with $\hat{a}_n=1$.
\begin{proof}
The red edge involves computing
\[\hat{a}_{n+1} = \floor{\frac{1}{1-x_n}}, \qquad x_{n+1}=\frac{1}{1-x_n}-\hat{a}_{n+1}.\]
Suppose the regular continued fraction expansion of some $x_n$ is given by
\[x_n = \cfrac{1}{k_1+\cfrac{1}{k_2+\cfrac{1}{k_3+\ddots}}}=[k_1,k_2,k_3,\ldots].\]
Then we see that
\begin{equation}\label{eqn:first appearance of insertion/deletion}
1-x_n = \begin{cases} [k_2+1,k_3,\ldots] & (k_1=1), \\
[1,k_1-1,k_2,\ldots] & (k_1 \neq 1). 
\end{cases}
\end{equation}
So if $k_1 = 1$ and we follow the red edge
\[ \hat{a}_1 = \floor{\frac{1}{1-x}} = k_2+1 \geq 2,\]
while if $k_1 \geq 2$ and we follow the red edge
\[\hat{a}_1 = \floor{\frac{1}{1-x}}=1,\] we will have
\[x_{n+1} = T(1-x) = [k_1-1,k_2,\ldots].\]
It is therefore only possible to follow the red edge with $\hat{a}_n=1$ at most $k_1$ consecutive times.
\end{proof}
\end{lemma}

Equation \eqref{eqn:first appearance of insertion/deletion} gives us a relation with what are called \emph{singularizations and insertions} of the standard continued fraction expansion of $x$.
These terms will be more precisely defined and elaborated upon in \Cref{section - relation to continued fractions}.

We now construct a natural Markov partition for $\hat{T}_{slow}$. The line $y=1-x$ is used to determine how $\hat{T}_{slow}$ acts, and then one must determine either $\floor{1/x}$ or $\floor{1/(1-x)}$, so the trapezoids bounded for $n=1,2,3,\ldots$ either by

\begin{itemize}
\item $1-x \leq y \leq 1$, $1/(n+1) \leq x \leq 1/n$, or
\item $0 \leq y \leq 1-x$, $(1-1/n) \leq x \leq (1-1/(n+1))$
\end{itemize}
form exactly that partition we are seeking; see Figure \ref{figure - markov partition for slow map}.
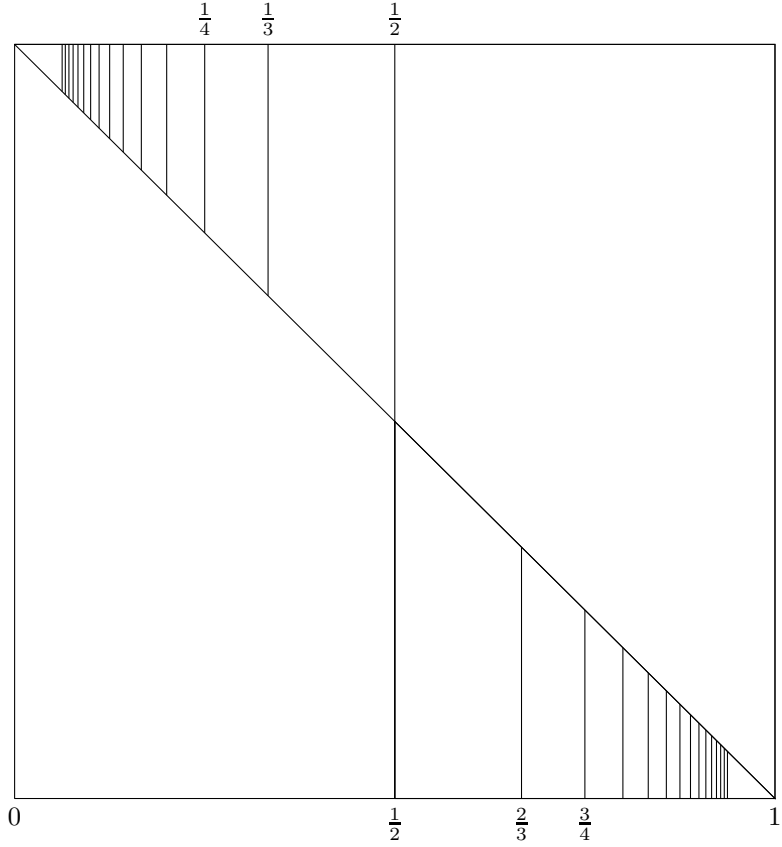
\begin{figure}[bht]
\center{\begin{tikzpicture}[scale=10]
\draw (0,0)node[below]{$0$}--(1,0)node[below]{$1$}--(1,1)--(0,1)--(0,0);
\draw (1/2,0)--(1,0)--(1/2,1/2)--(1/2,0);
\draw (0,1)--(1,1)--(1,0)--(0,1);
\draw (1/2,0)node[below]{ $\frac{1}{2}$};
\draw (2/3,0)node[below]{ $\frac{2}{3}$};
\draw (3/4,0)node[below]{ $\frac{3}{4}$};
\draw (1/2,1)node[above]{ $\frac{1}{2}$};
\draw (1/3,1)node[above]{ $\frac{1}{3}$};
\draw (1/4,1)node[above]{ $\frac{1}{4}$};

\foreach \j in {2,...,15}{%
	\draw ({1/(\j+1)},{1-1/(\j+1)})--({1/(\j+1)},1);
	}

\foreach \l in {1,...,15}{%
	\draw ({\l/(\l+1)},0)--({\l/(\l+1)},{1/(\l+1)});
	}
\draw (1/2,0)--(1/2,1);
\end{tikzpicture}
}

\caption{\label{figure - markov partition for slow map}A Markov partition for $\hat{T}_{slow}$.}
\end{figure}

The map $\hat{T}_{slow}$ may be extended to act on the segments $x=0$ and $x=1$: 
\[\hat{T}_{slow}(0,y) = (0,y), \qquad \hat{T}_{slow}(1,y)=(0,1-y).\]
As these segments have empty interior (in $[0,1]^2$), they are included to complete the partition but will have no effect on our dynamical statements later.

We ask, then, if the map $\hat{T}_{slow}$ is ergodic with respect to a probability measure mutually absolutely continuous with respect to Lebesgue measure. Observe, however, that for $x$ close to zero, we have $1/(1-x)-1$ also small; the map $\hat{T}_{slow}$ acts ``more and more like identity" the closer we get to $x=0$. Since, $\frac{d}{dx}(1/(1-x)-1)=\frac{1}{(1-x)^2}$, for $x=0$ any $y\in[0,1]$ is an indifferent fixed point of the map $\hat{T}_{slow}$. This observation suggests that $\hat{T}_{slow}$ does not preserve any \emph{finite} measures mutually absolutely continuous to Lebesgue measure. We will indeed eventually show this result (\Cref{corollary - T_slow not ergodic}), but for now we create a ``sped-up" version of $\hat{T}_{slow}$ to avoid this problem.

Define $F = \left\{(x,y):y \geq 1-x \quad \textrm{or} \quad x>1/2\right\}$ and $N_F(x,y)=\min \{ n \geq 0 : \hat{T}_{slow}^n(x,y)\in F \}$.
The fast map is then given by $\hat{T}_{fast}(x,y)=\hat{T}_{slow}^{N_F(x,y)+1}(x,y)$. It follows from \Cref{lemma - fast returns defined} that this map is well-defined: we apply $\hat{T}_{slow}$ until we either have $y \geq 1-x$ (in which case we will follow the green edge) or until we have $y \leq 1-x$ but with $x>1/2$, in which case following the red edge will have $\hat{a} \geq 2$. In other words, the map $\hat{T}_{fast}$ iterates $\hat{T}_{slow}$ until entering into $F$, then applies $\hat{T}_{slow}$ one more time. Equivalently, all consecutive occurrences of `red edge with $\hat{a}=1$' under $\hat{T}_{slow}$ are absorbed into the next occurrence of `green edge, or red edge with $\hat{a} \geq 2$.'

We begin the study of $\hat{T}_{fast}$ with a lemma which describes the action of $\hat{T}_{slow}$ on the interval $y \in [0,1-x]$ for $x<1/2$.

\begin{lemma}
\label{lemma - slow T acting on interval}
With $x<1/2$ fixed, $\hat{T}_{slow}$ linearly maps $\{x\} \times [0,1-x]$ onto $\{T(1-x)\} \times [0,1]$, $T(1-x) = x/(1-x)$, and $\floor{1/T(1-x)}=\floor{1/x}-1$.
\begin{proof}
That $\hat{T}_{slow}$ acts linearly on the second coordinate is immediate: when $y \leq 1-x$, $\hat{T}_{slow}(x,y) = (T(1-x),y/(1-x))$, and $x$ is fixed. When $x < 1/2$ we have $\lfloor \frac{1}{1-x}\rfloor =1$ which gives
\[
T(1-x)=\frac{1}{1-x} - \floor{\frac{1}{1-x}}=\frac{1}{1-x}-1= \frac{x}{1-x}.
\]
The last claim follows from reciprocating the above.
\end{proof}
\end{lemma}

So in the event that $x<1/2$ and $y \notin [1-x,1]$ (\textit{e.g.} exactly the points in $[0,1]^2$ where $\hat{T}_{fast}$ is not just defined as $\hat{T}_{slow}$), \Cref{figure - slow map details} shows the action of $\hat{T}_{slow}$ on all $y \in [0,1-x]$. Specifically, if $1-(i+1)x < y < 1-ix$, then $1-ix_s < y_s<1-(i-1)x_s$, where $\hat{T}_{slow}(x,y) = (x_s,y_s)$.

\begin{figure}[bht]
\center{\begin{tikzpicture}[xscale=15]
\draw[|-|] (0,1)--(.05,1);
\draw[|-|] (.05,1)--(.15,1);
\draw[|-|] (.15,1)--(.25,1);
\draw[|-|] (.25,1)--(.7,1);
\draw[|-|] (.7,1)--(.8,1);
\draw[|-|] (.8,1)--(.9,1);
\draw[|-|] (.9,1)--(1,1);
\draw[|-|] (0,0)--(.08,0);
\draw[|-|] (.08,0)--(.2,0);
\draw[|-|] (.2,0)--(.32,0);
\draw[|-|] (.32,0)--(.76,0);
\draw[|-|] (.76,0)--(.88,0);
\draw[|-|] (.88,0)--(1,0);
\draw[dotted] (0,1)--(0,0);
\draw[dotted] (.05,1)--(.08,0);
\draw[dotted] (.15,1)--(.2,0);
\draw[dotted] (.25,1)--(.32,0);
\draw[dotted] (.7,1)--(.76,0);
\draw[dotted] (.8,1)--(.88,0);
\draw[dotted] (.9,1)--(1,0);
\node[left] at (0,1){$0$};
\node[above] at (.05,1){$1-ax$};
\node[above] at (.15,1.4){$1-(a-1)x$};
\node[above] at (.25,1){$1-(a-2)x$};
\node[above] at (.7,1){$1-3x$};
\node[above] at (.8,1){$1-2x$};
\node[above] at (.9,1){$1-x$};
\node[right] at (1,1){$1$};
\node[left] at (0,0){$0$};
\node[below] at (.08,0){$1-(a-1)x_s$};
\node[below] at (.2,-.4){$1-(a-2)x_s$};
\node[below] at (.32,0){$1-(a-3)x_s$};
\node[below] at (.76,0){$1-2x_s$};
\node[below] at (.88,0){$1-x_s$};
\node[right] at (1,0){$1$};
\end{tikzpicture}
}

\caption{\label{figure - slow map details}The action of $\hat{T}_{slow}$ on the second coordinate when $x<\frac{1}{2}$ and $y\in [0,1-x]$. Here $a=\floor{1/x} \geq 2$ and $x_s=T(1-x)$.}
\end{figure}
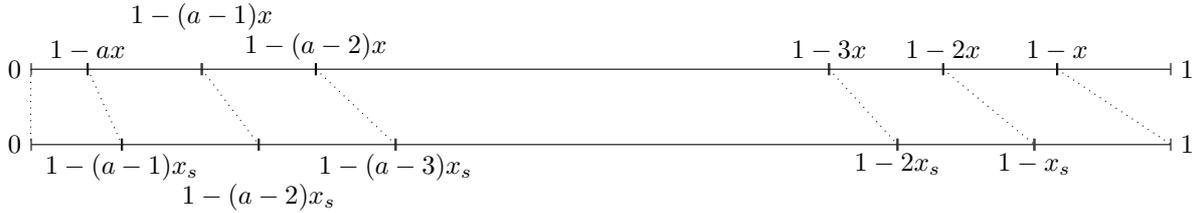

\begin{corollary}
For $x<1/2$ and $y<1-x$, partition $[0,1]$ with the points $1-ix$ for $i=1,2,\ldots,a$, where $a = \floor{1/x}$. For each $i = 1, 2,\ldots,(a-1)$, $\hat{T}_{fast}$ maps $\{x\} \times [1-(i+1)x,1-ix]$ linearly in its second coordinate onto $\{T(x)\}\times [0,1]$ with a reversal of orientation in the second coordinate. Also, $\hat{T}_{fast}$ maps $\{x\} \times [0,1-ax]$ linearly in its second coordinate onto $\{T^2(x)\} \times [0,1]$ with no reversal of orientation in the second coordinate. 
\label{corollary - basis of second graph}
\begin{proof}
The proof follows from viewing iterations of $\hat{T}_{slow}$ in light of \Cref{lemma - slow T acting on interval}. If $y \in [1-(i+1)x,1-ix]$, then after $i$ applications of $\hat{T}_{slow}$ ``following the red edge" (in the presentation of \Cref{figure - graph presentation}), we will arrive into $y$ belonging to the top-most interval, i.e. eventually $y \geq 1-x$, and we will ``follow the green edge." The assumption that $y \geq 1-ax$ ensures that every time we follow the red edge we do so with $\floor{1/(1-x)}=1$, and furthermore our rotation after applying $\hat{T}_{slow}$ exactly $i$ consecutive such times will be given by
\[x^* = \frac{1}{(a-i)+T(x)},\]
and in following the green edge we will now have rotation by $T(x^*)=T(x)$. 

The situation for $y \in [0,1-ax]$ is similar, except that we only ``follow the red edge." After $a-1$ iterations of $\hat{T}_{slow}$, the rotation amount is given by
\[ x^*=\frac{1}{a-(a-1)+T(x)} = \frac{1}{1+T(x)},\]
and then in the final step we compute
\begin{align*}
    1-x^* &= \frac{T(x)}{1+T(x)}\\
    \frac{1}{1-x^*}&= 1+ \frac{1}{T(x)}\\
    T(1-x^*)&=T(T(x)).
\end{align*}
The comments regarding orientation are simply due to considering how many times $\hat{T}_{slow}$ would involve a reversal of orientation: exactly once for those $y \geq 1-ax$ (computing $\hat{T}_{fast}$ will terminate in one application of a ``green edge"), and none for those $y \leq 1-ax$ ($\hat{T}_{fast}$ will instead terminate with a ``red edge with $\hat{a} \geq 2$").
\end{proof}
\end{corollary}

\Cref{corollary - basis of second graph} has another interpretation: after partitioning $[0,1]$ using the points $1-x,1-2x,\ldots,1-ax$ (where $a=\floor{1/x}$ regardless of $y$), the map $\hat{T}_{fast}$ can be considered to simply be the first return map on the interval of this partition containing $y$. When this interval is of length $x$, the induced rotation will be of length $T(x)$. When this interval is of length $1-ax$, the induced rotation will be of length $T^2(x)$. Then, $y$ is linearly scaled accordingly within this interval, reversing orientation when the interval is of length $x$, but not when the interval is of length $1-ax$. Altogether, then, we have the following:
\begin{equation}\label{eqn:T-hat-fast}
\hat{T}_{fast}(x,y) = \begin{cases} \left(T(x), \frac{1-y}{x} \mod 1 \right) & y \geq 1-\floor{1/x}x, \\
\left(T^2 (x), \frac{y}{1-\floor{1/x}x} \right) & y \leq 1-\floor{1/x}x . \end{cases}\end{equation}
We construct a graph presentation similar to \Cref{figure - graph presentation} for the action of $\hat{T}_{fast}$ in \Cref{figure - second graph presentation}.

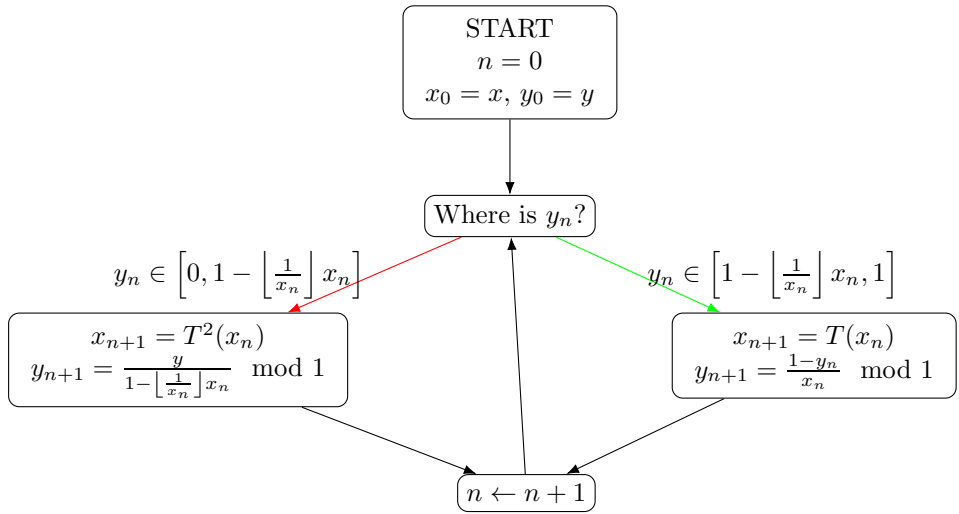
\begin{figure}[bth]
\center{\begin{tikzpicture}[xscale=1]
\node(start)[draw,rectangle, rounded corners]{$\begin{array}{c}\textrm{START}\\n=0\\x_0=x, \, y_0=y\end{array}$};
\node(set a)[below = of start, draw, rectangle, rounded corners]{Where is $y_n$?};
\node(regular)[below right = of set a, draw, rectangle, rounded corners]{$\begin{array}{c} x_{n+1} = T(x_n)\\y_{n+1} = \frac{1-y_n}{x_n} \mod 1 
\end{array}$};
\node(reversing)[below left = of set a, draw, rectangle, rounded corners]{$\begin{array}{c} x_{n+1} = T^2(x_n)\\y_{n+1} = \frac{y}{1-\floor{\frac{1}{x_n}}x_n} \mod 1
\end{array}$};
\node(increment)[below left = of regular,draw, rectangle, rounded corners]{$n \gets n+1$};
\draw[-Latex] (start) -- (set a);
\draw[-Latex,green] (set a) -- (regular) node[black,midway,right]{$y_n \in \left[1-\floor{\frac{1}{x_n}}x_n,1\right]$};
\draw[-Latex,red] (set a) -- (reversing) node[black,midway,left]{$y_n \in \left[0,1-\floor{\frac{1}{x_n}}x_n\right]$};
\draw[-Latex] (regular) -- (increment);
\draw[-Latex] (reversing) -- (increment);
\draw[-Latex] (increment) -- (set a);
\end{tikzpicture}}
\caption{\label{figure - second graph presentation}The graph version of $\hat{T}_{fast}$.
}
\end{figure}

We now describe a natural Markov partition for $\hat{T}_{fast}$. Since $\hat{T}_{fast}$ acts the same as $\hat{T}_{slow}$ on $F$, we use the same trapezoids in this region as we did for $\hat{T}_{slow}$. Specifically, we notate them as $G_n$ and $R_n$ as trapezoids bounded by the given lines:
\begin{align*}
 G_n: &\left\{ \frac{1}{n+1} \leq x \leq \frac{1}{n}, \quad 1-x \leq y \leq 1\right\}, \qquad n \geq 1,\\
 R_n: &\left\{1-\frac{1}{n+1} \leq x \leq 1-\frac{1}{n+2}, \quad 0 \leq y \leq 1-x\right\}, \qquad n \geq 1.
\end{align*}

In \Cref{figure - graph presentation} the trapezoids $G_n$ correspond to ``follow the green edge with $\hat{a}_{k+1}=n$", while $R_n$ correspond to ``follow the red edge with $\hat{a}_{k+1}=n+1$." The index $k+1$ here refers to a generic ``next index" and is non-specific. Also, for all $x \in G_n$, the standard continued fraction representation of $x$ begins with $x=[n,\ldots]$, while for $x \in R_n$ the standard continued fraction representation begins with $x=[1,n,\ldots]$. For convenience, denote $G=\cup_i G_i$ and $R=\cup_i R_i$: $G$ is the triangle bounded by $y=1$, $x=1$, and $y=1-x$, while $R$ is the triangle bounded by $x=1/2$, $y=0$, and $y=1-x$. 
The remaining trapezoid will be partitioned in the following way
\begin{align*}
 PG_{i,n}: &\left\{ \frac{1}{n+1} \leq x \leq \frac{1}{n}, 1-(i+1)x \leq y \leq 1-ix\right\}, \qquad n \geq 2, \quad i=1,\ldots, n-1,\\
 T_n: &\left\{\frac{1}{n+1} \leq x \leq \frac{1}{n}, \quad 0 \leq  y \leq 1-nx\right\}, \qquad n \geq 2.
\end{align*}
For consistent labeling, we also let $PG_{0,n}=G_n$ and $T_1=R$.

Directly from \Cref{lemma - slow T acting on interval} we get that the map $\hat{T}_{slow}$ acts on the trapezoids $PG_{i,n}$ as follows:
\begin{itemize}
\item For $i=1,2,\ldots,n-1$ and $n\geq 2$, $\hat{T}_{slow}: PG_{i,n} \mapsto PG_{i-1,n-1}$.
\item For $n \geq 2$, $\hat{T}_{slow}: T_n \mapsto T_{n-1}$.
\end{itemize}
All of the maps as written are $1:1$.

So we see that each of the trapezoids $PG_{i,n}$ will orbit under $\hat{T}_{slow}$ into $G_{n-i}$ after exactly $i$ applications of $\hat{T}_{slow}$, where $i \leq n-1$. If we update our notation to set $G_n = PG_{0,n}$, then we have the trapezoids $PG_{i,n}$, with $0 \leq i < n$, are acted on by $\hat{T}_{fast}$ by $i$ applications of $\hat{T}_{slow}$ of ``red edge with $\hat{a}_{k+1}=1$" followed by a single instance of ``green edge with $\hat{a}_{k+1}=(n-i)$." The notation ``$PG$" for trapezoids refers to them being ``pre-green." 

In contrast, the triangles $T_n$ will map into $R$ after $n-1$ applications of $\hat{T}_{slow}$, at which point $\hat{T}_{slow}$ acts as ``red edge, but with $\hat{a}_{k+1} \geq 2$". So we accordingly partition each $T_n$ into which trapezoid $R_n$ it will orbit into. More specifically, let the trapezoids $PR_{i,n}$ be bounded by the lines
\[ 
PR_{i,n}: \left\{ 
\frac{i}{in+1} \leq  x \leq  \frac{i+1}{(i+1)n+1}, \quad 0 \leq y \leq 1-nx. \right\} \qquad n \geq 2,\quad i=1,\ldots, n-1. 
\]

Note that $\hat{T}_{slow}: PR_{i,n} \mapsto PR_{i,n-1}$, which is consistent if we relabel our original $R_i = PR_{i,0}$. The trapezoid $PR_{i,n}$ is acted on by $\hat{T}_{slow}$ as ``$n-1$ red edges with $\hat{a}_{k+1}=1$, followed by a single red edge with $\hat{a}_{k+1}=i+1$." These trapezoids are accordingly labeled as ``pre-red." In summary:

\begin{itemize}
    \item The trapezoid $PG_{i,n}$ refers to a region in which $\hat{T}_{fast}$ is given by $i$ iterations of $\hat{T}_{slow}$ with ``red edge, $\hat{a}_{k+1}=1$" followed by one iteration of $\hat{T}_{slow}$ with ``green edge, $\hat{a}_{k+1}=n-i$." Here $0 \leq i \leq n-1$.
    \item In $PG_{i,n}$, $n$ is the first partial quotient of $x$, while $i$ refers to information about the $y$ coordinate: $1-(i+1)x \leq y \leq 1-ix$.
    \item In contrast, the trapezoid $PR_{i,n}$ refers to a region in which $\hat{T}_{fast}$ is given by $n-1$ iterations of $\hat{T}_{slow}$ with ``red edge, $\hat{a}_{k+1}=1$" followed by one iteration of $\hat{T}_{slow}$ with ``red edge, $\hat{a}_{k+1} \geq 2$." Here $1 \leq i < \infty$.
    \item In $PR_{i,n}$, the $x$-coordinates are defined by $i/(ni+1) \leq x \leq (i+1)/(n(i+1)+1)$; the first two partial quotients of $x$ are $n,i$, while $y$ must satisfy $y \leq 1-nx$.
\end{itemize}

See \Cref{figure - markov partition for fast map} to see this Markov partition; arrows show how successive applications of $\hat{T}_{slow}$ eventually map trapezoids into either $G$ or $R$, where they are mapped bijectively back to the unit square. So $\hat{T}_{fast}$ bijectively maps each of these trapezoids to the unit square. The green trapezoids account for those $y$ which are not ``very small" by partitioning the $x$ into cylinders of length one for the standard continued fraction expansion, while the red trapezoids account for those smaller values of $y$, but at the expense of partitioning the $x$ into cylinders of length two for the standard continued fraction expansion.

\begin{figure}[bht]
\center{\begin{tikzpicture}[scale=10]
\draw (0,0)--(1,0)--(1,1)--(0,1)--(0,0);
\draw [fill=red] (1/2,0)--(1,0)--(1/2,1/2)--(1/2,0);
\draw [fill=green] (0,1)--(1,1)--(1,0)--(0,1);
\foreach \j in {2,...,15}{%
	\draw [fill=lime](1/\j,0)--(1/\j,1-1/\j)--({1/(\j+1)},{1-1/(\j+1)})--({1/(\j+1)},{1/(\j+1)})--(1/\j,0);}
\foreach \j in {2,...,15}{%
	\draw [fill=pink]({1/(\j+1)},0)--(1/\j,0)--({1/(\j+1)},{1/(\j+1)})--({1/(\j+1)},0);
	\draw ({1/(\j+1)},0)--({1/(\j+1)},1);
	\foreach \k in {1,...,15}{%
	\draw ({((\k)/(\j*\k+1)},0)--({((\k)/(\j*\k+1)},{((1)/(\j*\k+1)});
	}
	\draw (0,1)--(1/\j,0);
	}
\foreach \l in {1,...,15}{%
	\draw ({\l/(\l+1)},0)--({\l/(\l+1)},{1/(\l+1)});
	}
\draw (1/2,0)--(1/2,1);
\foreach \x in {1,...,5}{
	\foreach \y in {0,...,\x}{
		\draw[dotted, -stealth] ({(4*(\x+1)+1)/(4*(\x+1)*(\x+2))},{(2*(\x+1)*\y+1)/(2*(\x+1)*(\x+2))})--({(4*\x+1)/(4*\x*(\x+1))},{(2*\x*\y+1)/(2*\x*(\x+1))});
		}}
\node at(.5, 1.05){$G=\cup_n PG_{0,n}$};
\node at(.75,.75){$PG_{0,1}$};
\node at(.4,.80){$PG_{0,2}$};
\node at(.4,.5){$PG_{1,2}$};
\node at(.29,.85){\small $PG_{0,3}$};
\node at(.29,.6){\small $PG_{1,3}$};
\node at(.29,.35){\small $PG_{2,3}$};
\node at(.3,-.05){$T_3=\cup_i PR_{i,3}$};
\node at(.4,-.1){$T_2=\cup_i PR_{i,2}$};
\node at(.75,-.05){$R=\cup_i PR_{i,1}$};
\node at(.6,.3){$PR_{1,1}$};
\node at(.71,.22){\footnotesize $PR_{2,1}$};
\node at(.78,.16){\scriptsize $PR_{3,1}$};
\end{tikzpicture}
\caption{\label{figure - markov partition for fast map}A Markov partition for $\hat{T}_{fast}$ showing the action of $\hat{T}_{slow}$.}}
\end{figure}
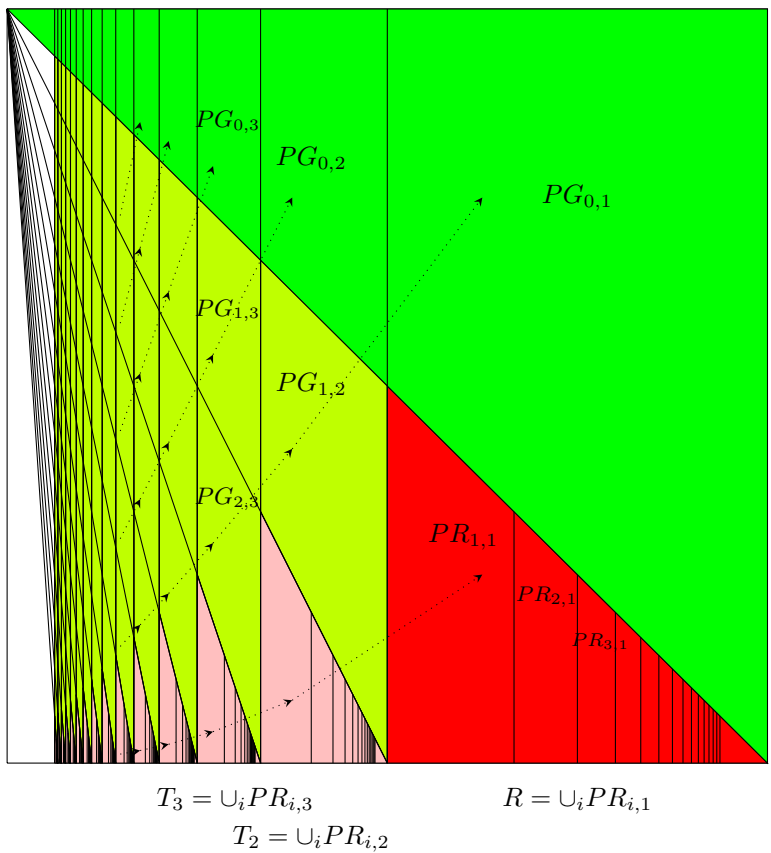

We now present the central claim of this section:
\begin{theorem}
There is a probability measure $\hat{\mu}$ on the space $[0,1]^2$, mutually absolutely continuous with respect to Lebesgue measure $\mu$, for which the map $\hat{T}_{fast}$ is ergodic.
\label{theorem - T_fast is ergodic}
\begin{proof}
Recall that formally the second coordinate contains left/right separation of points $y=nx$; we take $\mu$ to be the natural extension of Lebesgue measure to this space. In order to prove our theorem we will show that the conditions in \cite[Theorem 1.1]{1078-0947_2005_4_639} are satisfied. First, we note that our system
\[ \left([0,1]^2,\hat{T}_{fast},\mu \right)\]
is a \emph{tower system}. The `base' of our tower system is the space $[0,1]^2$, with all points having return time one (i.e. the towers are all of height one, as $\hat{T}_{fast}$ maps this space to itself).
There are four conditions to verify, and we provide the terminology of our reference for easy comparison, though we do not necessarily redefine all terms herein:\\
\emph{First: ``summability of upper floors"}

With the return time $R(x,y)=1$ for all $(x,y)$ we immediately verify
\begin{align*}
    \sum_{\ell \in \mathbb{N}} \mu \left( \left\{ (x,y) \in [0,1]^2  \quad | \quad R(x,y) > \ell \right\} \right) &=\mu \left\{ (x,y) \in [0,1]^2 \quad | \quad R(x,y) =1 \right\}+0\\
    &=\mu([0,1]^2)\\
    &= 1.
\end{align*} 
\emph{Second: the areas $PG_{i,n}$, $PR_{i,n}$ form a generating partition $\mathcal{R} $.}

This condition follows from a verification that some power of $\hat{T}_{fast}$ is \emph{uniformly expanding} (has norm bounded below by a number larger than one). We revisit \Cref{eqn:T-hat-fast} as:
\begin{equation}\label{eqn:T-hat-fast2}
\hat{T}_{fast}(x,y) = 
\begin{cases} \left( \frac{1}{x} - n, \frac{1-y}{x} - i\right) &(x,y) \in PG_{i,n}, \\[0.5em]
\left(\frac{x(ni+1)-i}{1-nx}, \frac{y}{1-nx}\right)&(x,y) \in PR_{i,n},  
\end{cases}
\end{equation}
from which we can more easily compute the Jacobian $D\hat{T}_{fast}$:
\[
D\hat{T}_{fast}(x,y)  = 
\begin{cases} \left[\begin{array}{cc} - \frac{1}{x^2} & 0 \\ -\frac{(1-y)}{x^2} & - \frac{1}{x} \end{array}\right] & (x,y) \in PG_{i,n}, \\[10pt]
\left[ \begin{array}{cc} \frac{1}{(1-nx)^2} & 0 \\ \frac{ny}{(1-nx)^2} & \frac{1}{(1-nx)} \end{array}\right] & (x,y) \in PR_{i,n}. 
\end{cases}
\]

As these matrices are triangular with distinct real numbers along the diagonal, the smallest modulus of an eigenvalue is easy to compute as either $1/x$ or $1/(1-nx)$. The latter is bounded away from one, so $\hat{T}_{fast}$ is uniformly expanding on $PR_{i,n}$. On $PG_{i,n}$, however, this modulus is not bounded away from one. Also note that in both cases there are two distinct real eigenvalues, both of which are larger than one in absolute value; the map can never \emph{contract} anywhere.

But we need only verify that \emph{some power} of $\hat{T}_{fast}$ is uniformly expanding, so we consider $\hat{T}^2_{fast}$. We have already verified that $\hat{T}_{fast}^2$ is uniformly expanding on any $PR_{i,n}$ or $\hat{T}_{fast}^{-1} PR_{i,n}$ (it is uniformly expanding on $PR_{i,n}$ and nowhere-contracting in the $PG_{i,n}$), so consider instead only the areas $PG_{i_1,n_1} \cap \hat{T}^{-1}_{fast}\left(PG_{i_2,n_2}\right)$. For such $(x,y)$, the continued fraction expansion of $x$ begins $x=[n_1,n_2,\ldots]$, and via the chain rule:
\begin{align*}
 D\left(\hat{T}_{fast}^2(x,y)\right) &= D\hat{T}_{fast}(x,y) \cdot D\hat{T}_{fast} \left( \hat{T}_{fast}(x,y)\right)\\
&= \left[\begin{array}{cc}-\frac{1}{x^2} &0 \\ \star &-\frac{1}{x}  \end{array} \right] \cdot \left[\begin{array}{cc} - \frac{1}{T(x)^2} &0 \\ \star & \frac{1}{T(x)} \end{array}\right]\\
&= \left[ \begin{array}{cc}\frac{1}{(xT(x))^2} &0 \\ \star & \frac{1}{xT(x)} \end{array}\right].
\end{align*}
We did not explicitly compute the lower-left entries because they are not relevant for computing the norm of these triangular matrices: in these regions we now see that the smaller modulus of an eigenvalue is given by $1/(x T(x)) > 2$. Therefore $\hat{T}^2_{fast}$ is uniformly expanding, and our partition generates the topology of our space.\\
\emph{Third: summable variation}

The third condition is that
\[ \sum_{n \in \mathbb{N}} \omega_n < \infty,\]
where $\omega_n$ is given by
\[
\omega_n= \sup_{C\in\mathcal{R}^n} \, \sup_{(x_1,y_1),(x_2,y_2) \in C} \log \frac{\det D\hat{T}_{fast}(x_1,y_1)}{\det D\hat{T}_{fast}(x_2,y_2)},
\]
with $\mathcal{R}^n=\bigvee_{i=0}^{n-1}\hat{T}^{-i}_{fast}(\mathcal{R})$. We see that the determinant is given by
\[
\det \left(D\hat{T}_{fast}(x,y)\right) = \begin{cases}x^{-3} & (x,y) \in PG_{i,n}, \\
(xT(x))^{-3} & (x,y) \in PR_{i,n}. \end{cases}
\]

Observe that this determinant is independent of $y$. Furthermore, the areas $PG_{i,n}$ are defined through a length one cylinder in the standard continued expansion of $x$, and then $x$ is mapped to $T(x)$. The $PR_{i,n}$ are defined through a length two cylinder in the same way, where $x$ is mapped to $T^2(x)$. Therefore the cylinders of $\hat{T}^n_{fast}$ are bounded in the $x$-coordinate by length $k$ cylinders in the standard continued fraction expansion, where $n \leq k \leq 2n$, depending on how many times we applied $T$ versus $T^2$. 

So our $\omega_n$ are in fact a supremum over all $(a_1,a_2,\ldots,a_k)\in \mathbb{Z}^k$, over all $x_1,x_2$ whose continued fraction expansion begins with $[a_1,a_2,\ldots,a_k]$, of the quantity
\[ \max\left\{3\log \left( \frac{x_1}{x_2}\right),3\log \left( \frac{x_1T(x_1)}{x_2T(x_2)}\right)\right\}.\]
If $x_1,x_2$ belong to the same cylinder of length $k$, then $T(x_1),T(x_2)$ belong to the same cylinder of length $k-1$. Therefore it suffices to show summability only when finding the supremum of $\log \left( \frac{x_1}{x_2}\right)$. Furthermore, the supremum is found by taking $x_1$, $x_2$ to be the endpoints of the cylinders. The distance between the endpoints is given by $\frac{1}{q_k(q_k+q_{k-1})}$ where $q_k$ is the denominator of $[a_1,\ldots, a_k]$. So the supremum occurs for the cylinders with smallest denominators, \textit{i.e.} $a_1=a_2=\cdots=a_k=1$. 
We let $\varphi=[1,1,1,\ldots]$ so that $1/\varphi$ is the  golden mean. We find $x_1 = F_k/F_{k+1}$ and $x_2=F_{k+1}/F_{k+2}$ (for $k$ odd; otherwise switch the choice of $x_1$, $x_2$). Here $(F_k)_{k\in \mathbb{N}}$ is the Fibonacci sequence. The ratio is largest when $k$ is smallest, and since $n \leq k \leq 2n$, we let $k=n$. Finally, we have 
\[
\left| x_1 - \varphi\right| < \frac{1}{F_{n+1}^2}, \qquad  \left| x_2 - \varphi \right| < \frac{1}{F_{n+2}^2}.
\]

Regardless of whether $k$ is even or odd, from the above we may derive that for some $C<1$ (specifically for any $1>C> \varphi$) and for all sufficiently large $n$  we have
\[ \left| \frac{x_1}{x_2} - 1\right|< C^n, \qquad 1 < \frac{x_1}{x_2} < 1 + C^n, \qquad 0< \log \left( \frac{x_1}{x_2}\right) <C^n\] from which summability of the $\omega_n$ follows.\\
\emph{Fourth: Large Image Properties}

The fourth and final condition of \cite[Theorem 1.1]{1078-0947_2005_4_639} is that for each element of our partition, the image under $\hat{T}_{fast}$ is comprised of elements of our partition, and that the infimum of the measure of these images is positive. Since $\hat{T}_{fast}$ maps each element of the partition to the entire space $[0,1]^2$, this condition is trivially satisfied.
\end{proof}
\end{theorem}

The theorem we cite here also provides estimates on the rate of mixing for this preserved measure: as the $\omega_n$ are exponentially decaying, the mixing rate of the system is exponential.

\begin{corollary}\label{corollary - T_slow not ergodic}
The map $\hat{T}_{slow}$ does not preserve any absolutely continuous (with respect to Lebesgue) probability measures with bounded Radon-Nikodym derivative. Rather, $\hat{T}_{slow}$ preserves a measure $\mu$ which is mutually absolutely continuous with respect to Lebesgue measure, but for which $\mu\left([0,1]^2 \right)=\infty$.
\begin{proof}
The results of \cite{1078-0947_2005_4_639} additionally show that the Radon-Nikodym derivative $d\nu/d(x,y)$, where $\nu$ is the invariant probability measure for $\hat{T}_{fast}$ which is absolutely continuous with respect to Lebesgue measure $d(x,y)$, is both bounded and bounded away from zero. There is an additional \textit{aperiodicity} requirement in that reference which refers to the partition elements; since each element of our partition maps to the entire square, this condition is met.

As there is a straightforward presentation of $\hat{T}_{slow}$ as a tower over $\hat{T}_{fast}$, we see then that if $\hat{T}_{slow}$ preserved such a measure, then the $\nu$ measure of the of the base would be proportional (according to the bounds on the derivative $d\mu/d(x,y)$) to the sum of the $\mu$-measures of the towers. Since our towers are of height $n$ over a base of area on the order of $1/n^2$, no such finite measure can exist. If we simply define a measure $\mu$, however, to be on each level of the tower the $\nu$-measure of the base, then we will have constructed an infinite (but still $\sigma$-finite) measure which is preserved by $\hat{T}_{slow}$.
\end{proof}
\end{corollary}
\section{Canonical Approximations}\label{section - symbolic coding}

As we successively constructed first return maps in \Cref{sec:TslowTfast} we repeatedly rescaled the resulting interval to be of length one; this renormalization allowed us to construct the map $\hat{T}_{fast}$ on the space $[0,1]^2$. In this section, however, we keep continued track of the corresponding shrinking intervals containing $y$ in the original space. We simultaneously develop substitutions which encode of the orbits of the endpoints of these intervals through their first returns.

It is clear from the iterative construction of first-return maps in \Cref{sec:TslowTfast} that we have a natural sequence of intervals $I_{slow}(k)$ which are nested ($I_{k+1} \subset I_k$) with intersection exactly $\{y\}$. Specifically, $I_{slow}(0)=[0,1]$, and then if $I_{slow}(k)$ is partitioned by the first pre-image of the origin which appears in the interior, $I_{slow}(k+1)$ is which of the two resulting intervals contains $y$; recall that for $y \in x \mathbb{Z}$ we have left- and right-sided versions of $y$, so this choice is always unique. These are therefore the intervals in the original circle $[0,1]$ on which the maps $\hat{T}_{slow}$ successively computed the first-return map. Observe trivially that both endpoints of $I_{slow}(k)$ are of the form $-nx$ for some non-negative integer $n$. We define those natural numbers $n$ so that $-nx$ is an endpoint of some $I_{slow}(k)$, listed in increasing order, to be the \textit{slow approximating sequence of $y$ with respect to $x$}. We formalize these notions below:

\begin{lemma}
If $\{n_k\}$ are the slow approximating sequence of $y$ with respect to $x$, then for all $k \in \mathbb{N}$, $-n_k x$ and $-n_{k-1}x$ are the endpoints of $I_{slow}(k)$ (under the convention that $n_{-1}=n_0=0$). If we let $\{a,b\}=\{n_k,n_{k+1}\}$ so that $0 \leq y+ ax < y+ bx \leq 1 \mod 1$, then either $y+n_{k+2}x < y+ax \mod 1$ or $y+bx < y+n_{k+2}x \mod 1$. Furthermore, $n_{k+2}$ is the least positive integer which satisfies either of those inequalities.
\begin{proof}
Both endpoints of $I_{slow}(0)=[0,1]$ are of the form $-0x$, so the base case of the first claim is immediate. Our construction of the intervals $I_{slow}(k)$ is such that the endpoints form the sequence $-n_k x$ which are successively the closest points in the \textit{backwards} orbit of the origin which are closest to $y$. If the points $-n_k \mod 1$ are the successively closest points to $y$, then $y+n_k \mod 1$ are the successively closest points to the origin.
\end{proof}
\end{lemma}

There is therefore a subsequence of these intervals, $I_{fast}(k)$ which are the intervals on which the first return map is a scaled version of $\hat{T}_{fast}^k$; recall that 
\[\hat{T}_{fast}(x,y) = \begin{cases}\hat{T}_{slow}^{i+1}(x,y) & \left((x,y) \in PG_{i,n}\right), \\ \hat{T}_{slow}^{n}(x,y) & \left((x,y) \in PR_{i,n}\right). \end{cases} \]
So we construct a subsequence of $k$ beginning with $k_0=0$, and then
\[ k_{j+1} = \begin{cases}k_j+ i +1 & \left(\hat{T}_{fast}^j(x,y) \in PG_{i,n}\right), \\[3pt] k_j+n & \left( \hat{T}_{fast}^j(x,y) \in PR_{i,n}\right).\end{cases}\]
In this way, letting $I_{fast}(j) = I_{slow}(k_j)$, these are the intervals in the original system on which $\hat{T}_{fast}$ is constructing first-return maps.

We define those $n_{k_j}$ so that $-n_{k_j}x$ is an endpoint of some $I_{fast}(k_j)$ to be the \textit{fast approximating sequence of $y$ with respect to $x$}. These numbers, then, represent a ``sped-up" sequence of $n_k$ along which $y+n_{k_j}x \mod 1$ converges to the origin. What we have called the slow and fast approximating sequences of $y$ with respect to $x$ represent those points in the orbit of $y$ which are close to the origin: $-nx \mod 1$ is close to $y$, so $y+nx \mod 1$ is close to $0$ or $1$. What is more typical is to consider those points in the orbit of the origin which are close to $y$. In \cite{IN} these are called the \textit{canonical approximating sequence} for $y$ with respect to $x$. Our definition is similar but not identical; the difference amounts to approximating $y$ with points in the forward orbit of the origin (the canonical sequence) versus the backwards orbit of the origin (our fast approximating sequence). Therefore the fast approximating sequence for $y$ with respect to $x$ would be the canonical approximating sequence for $y$ with respect to $1-x$, and vice versa.

Let $\mathcal{A}=\{A,B\}$ and $\mathcal{A}^{*}$ be the free monoid on $\mathcal{A}$. Let the function $\Omega:[0,1]^2 \mapsto \mathcal{A}^{\mathbb{N}}$ be defined by
\[ \Omega(x,y)_n = \begin{cases} A & (y+nx \mod 1 \geq 1-x) \\ B & (y+nx \mod 1\leq 1-x) .\end{cases}\]
Then $\Omega(x,y)$ is called the \textit{symbolic encoding of the orbit of $y$ under rotation by $x$}, and the prefix $\Omega(x,y)_{[0,n)}$ is called the symbolic encoding \textit{of length $n$}. For a given pair of coordinates $(x,y)$, we will also simply refer to the \textit{symbolic coding of the pair $(x,y)$} in place of the symbolic coding of $y$ under rotation by $x$.

We define the homomorphisms $\sigma_x$, $\tau$ on $\mathcal{A}^*$ by:
\begin{align*}\label{eqn - basic sigmas}
\sigma_x(A) &= AB^a & \tau(A)&=B\\
\sigma_x(B) &= AB^{a-1} & \tau(B) &=A
\end{align*}
where $a= \floor{1/x}$ as before. These can be naturally extended to be homomorphisms on $\mathcal{A}^{\mathbb{N}}$. Then let
\begin{equation}\label{equation - slow sigma} \sigma_{slow}(x,y) = \begin{cases} \sigma_x & (y \geq 1-x) \\ \tau \circ \sigma_{1-x} & (y \leq 1-x).\end{cases}
\end{equation}
Representing any word in $\mathcal{A}^{*}$ with $n$ occurrences of $A$ and $m$ occurrences of $B$ by the column vector $[n, m]^{T}$, we see that the substitutions may be studied through their matrices $M_{slow}(x,y)$
\begin{equation}\label{eqn:slow matrices}
    M_{slow}(x,y) = 
    \begin{cases} 
    \left[ \begin{array}{cc} 1&1 \\\hat{a} & (\hat{a}-1) \end{array}\right] & (y \geq 1-x, \, \hat{a} = \floor{1/x}) \\[1em]
    \left[ \begin{array}{cc}\hat{a} & (\hat{a}-1)\\ 1& 1 \end{array}\right] & (y \leq 1-x, \, \hat{a} = \floor{1/(1-x)}). \\
    \end{cases}
\end{equation}
Of particular interest is when $y \leq 1-x$, and $1-x \geq 1/2$ (i.e. $\hat{a}=1$), in which case $\sigma(A) = BA$, $\sigma(B)=B$, i.e. those $x,y$ for which
\begin{equation}\label{eqn:special matrix} M_{slow}(x,y) = \left[ \begin{array}{cc} 1 & 0 \\ 1 & 1 \end{array}\right].\end{equation}
Note that 
\begin{equation}\label{eqn:special matrix powers}
\left[ \begin{array}{cc} 1 & 0 \\ 1 & 1 \end{array}\right]^d=
\left[ \begin{array}{cc} 1 & 0 \\ d & 1 \end{array}\right].
\end{equation}
Finally, define the ergodic composition recursively by
\[
\sigma_{slow}(x,y,k+1) = \sigma_{slow}(x,y,k) \circ \sigma_{slow}\left(\hat{T}^{k}_{slow}(x,y)\right)
\]
with the convention that $\sigma_{slow}(x,y,0)$ is the identity substitution, and with the matrices $M_{slow}(x,y,k)$ similarly defined. 

Recall that $\hat{T}_{fast}$ was equivalent to composing all occurrences $\hat{T}_{slow}$ where $y \leq 1-x$ and $1-x \geq 1/2$, the ``red edges with $\hat{a}_{k+1}=1$," into the next application of $\hat{T}_{slow}$. In $PG_{i,n}$ there are exactly $i$ such occurrences (with $0 \leq i \leq n-1$), and in $PR_{i,n}$ there are exactly $n-1$. In other words, we may define $\sigma_{fast}(x,y)$ as

\begin{equation}\label{equation - fast sigma}
\sigma_{fast}(x,y) = \begin{cases} \sigma_{slow}(x,y,i+1) & (x,y) \in PG_{i,n} \\ \sigma_{slow}(x,y,n) & (x,y) \in PR_{i,n}. \end{cases}
\end{equation}
We also define the matrices $M_{fast}(x,y)$, $M_{fast}(x,y,k)$ similarly. By combining \Cref{eqn:slow matrices}, \Cref{eqn:special matrix}, \Cref{eqn:special matrix powers}, \Cref{equation - fast sigma}:

\begin{equation*}
    M_{fast}(x,y,k) = 
    \begin{cases} 
    \left[ \begin{array}{cc} 1&1 \\ \hat{a} &(\hat{a}-1) \end{array}\right] \cdot \left[ \begin{array}{cc}1 &0 \\i &1 \end{array}\right]& \left((x_k,y_k) \in PG_{i,n}\right), \quad \hat{a}=\floor{1/x_n})\\[10pt]
    \left[ \begin{array}{cc}\hat{a} & (\hat{a}-1)\\1 &1 \end{array}\right] \cdot \left[ \begin{array}{cc}1 &0 \\(n-1) &1 \end{array}\right]& \left((x_k,y_k) \in PR_{i,n}, \quad \hat{a}=\floor{1/(1-x_n)}\right).\\
    \end{cases}
\end{equation*}
The values of $\hat{a}$ can be determined as follows: in $PG_{i,n}$, after $i$ consecutive ``red edge with $\hat{a}_{k+1}=1$" applications of $\hat{T}_{slow}$, we will ``follow the green edge" with $\hat{T}_{slow}$ but where the first partial quotient is $n-i$. So in $PG_{i,n}$, we use $\hat{a}=n-i$. In the $PR_{i,n}$, after $n-1$ applications of ``red edge with $\hat{a}_{k+1}=1$" we will follow the red edge for a rotation whose standard continued fraction begins with $x'=[1,i,\ldots]$, so $\hat{a}= \floor{1/(1-x')} = i+1$.
Then $M_{fast}(x,y,k)$ is the matrix of the substitution $\sigma_{fast}(x_k,y_k)$; the single substitution generated after $k$ iterations of $\hat{T}_{fast}$. Altogether:
\begin{equation}
    \label{eqn:matrix powers computed}
      M_{fast}(x,y) = 
    \begin{cases} 
    \left[ \begin{array}{cc} i+1&1 \\((n-i)(i+1)-i) &(n-i-1) \end{array}\right]& \left((x_k,y_k) \in PG_{i,n}\right)\\[10pt]
    \left[ \begin{array}{cc}(ni+1)& i\\n &1 \end{array}\right]& \left((x_k,y_k) \in PR_{i,n}\right).\\
    \end{cases}
\end{equation}
These matrices have two distinct real eigenvalues, one positive and one negative: define $\lambda_{\pm}(x,y)$ to be these eigenvalues. From \eqref{eqn:matrix powers computed} we directly compute
\begin{align*}
    &\lambda_{\pm}=\frac{n \pm \sqrt{n^2+4}}{2} & \text{ for } (x,y) \in PG_{i,n} \\
    &\lambda_{\pm} = \frac{(ni+2) \pm \sqrt{(ni+2)^2+4}}{2} & \text{ for } (x,y) \in PR_{i,n}. 
\end{align*}
We can now state:
\begin{theorem}\label{theorem - lambda integrable}
Both $\log^+\|M_{fast}(x,y)\|$ and $\log^+\|M_{fast}^{-1}(x,y)\|$ are Lebesgue integrable on $[0,1]^2$, where $\log^+(t)=\max\{0,\log(t)\}$ and the norm is the operator norm.
\begin{proof}
We begin with several claims about these norms:
\begin{enumerate}
    \item In either $PG_{i,n}$ or $PR_{i,n}$, both of $\|M_{fast}\|, \|M_{fast}^{-1}\| \geq 1$. That is: the $\log^+$ function will always simply evaluate as a logarithm.
    \item In $PG_{i,n}$, $\|M_{fast}\|\leq n+1$, while in $PR_{i,n}$, $\|M_{fast}\| \leq ni+3$.
    \item In $PG_{i,n}$, $\|M_{fast}^{-1}\| \leq 2n$, while in $PR_{i,n}$, $\|M_{fast}^{-1}\| \leq 2(ni+2)$.
\end{enumerate}
The matrix $M_{fast}(x,y)$ has two distinct eigenvalues of the form
\[ \lambda_{\pm} = \frac{\zeta \pm \sqrt{\zeta^2+4}}{2},\]
where $\zeta$ is a positive integer. Specifically, $\zeta = n$ in $PG_{i,n}$ and $\zeta = ni+2$ in $PR_{i,n}$. It follows that
\begin{equation}\label{eigMfast}
\|M_{fast}(x,y)\| = \frac{\zeta+\sqrt{\zeta^2+4}}{2}, \qquad \|M_{fast}^{-1}(x,y)\| = \frac{2}{\sqrt{\zeta^2+4}-\zeta}.    
\end{equation}
The inequality $\zeta >0$ establishes that these norms are at least one and completes our first claim.

For the upper bound on $\|M_{fast}(x,y)\|$, the elementary estimate $\zeta^2+4 <(\zeta+2)^2$ yields the desired result. For the upper bound on the norm of the inverse operator, $(\zeta+1/\zeta)^2<\zeta^2+4$ shows that $\|M_{fast}^{-1}(x,y)\|< 2\zeta$: $\zeta=n$ or $\zeta=ni+2$ in the $PG_{i,n}$ and $PR_{i,n}$ respectively complete this step.

We now proceed with the proof of integrability. The set $PG_{i,n}$ is a trapezoid with width $1/n-1/(n+1)<1/n^2$ and longer height $1/n$. The set $PR_{i,n}$ is a trapezoid with width $(i+1)/(n(i+1)+1)-i/(ni+1)<1/(ni+1)^2$ and longer height $1/(ni+1)$. So
\[ \textrm{Area}(PG_{i,n})< \frac{1}{n^3},\]
\[ \textrm{Area}(PR_{i,n})< \frac{1}{(ni+1)^3}< \frac{1}{(ni)^3}.\]
Recall that our Markov partition is given by $PG_{i,n}$ for $n=1,2,\ldots$ and $0 \leq i \leq n-1$, and $PR_{i,n}$ for $i,n=1,2,\ldots$. With our earlier claims, then:
\begin{align*}
    \int_{[0,1]^2} \log^+\|M_{fast}(x,y)\| dxdy &< \sum_{n=1}^{\infty} \sum_{i=0}^{n-1} \int_{PG_{i,n}}\log(n+1)dxdy + \sum_{i=1}^{\infty} \sum_{n=1}^{\infty} \int_{PR_{i,n}} \log(ni+3)dxdy\\
    &< \sum_{n=1}^{\infty}\sum_{i=0}^{n-1} \frac{\log(n+1)}{n^3}+\sum_{i=1}^{\infty}\sum_{n=1}^{\infty} \frac{\log(ni+3)}{(ni)^3}\\
    &< \sum_{n=1}^{\infty}\sum_{i=0}^{n-1} \frac{\log(n+1)}{n^3}+\sum_{i=1}^{\infty}\sum_{n=1}^{\infty} \frac{\log((n+2)(i+2))}{(ni)^3}\\
    &< \sum_{n=1}^{\infty} \frac{\log(n+1)}{n^2} + \left(\sum_{i=1}^{\infty}\frac{1}{i^3} \right) \left(\sum_{n=1}^{\infty} \frac{\log(n+2)}{n^3} \right)+\left(\sum_{n=1}^{\infty}\frac{1}{n^3} \right) \left(\sum_{i=1}^{\infty} \frac{\log(i+2)}{i^3} \right).
\end{align*}
Summability of all terms is elementary. Integrability of $\log ^+\|M_{fast}^{-1}\|$ is handled analogously; observe the similarity in the bounds derived.
\end{proof}
\end{theorem}

\Cref{theorem - lambda integrable} will be one of our primary tools in studying the growth rate of the slow approximating sequence and fast approximating sequence. We again begin with the slow situation before moving on the fast. Observe that between $I_{slow}(n)$ and $I_{slow}(n+1)$, one endpoint is always shared between these two intervals, and one endpoint is not. Define a sequence of $\epsilon_{slow,n}$ to track the number of times that orientation has been reversed in applying $\hat{T}_{slow}$:
\[ \epsilon_{slow,0}=1, \qquad \epsilon_{slow,n+1} = \begin{cases} -\epsilon_{slow,n} & \hat{T}_{slow}^n(x,y) \in G \\ \epsilon_{slow,n} & \hat{T}_{slow}^n(x,y) \in R. \end{cases}\]
That is: we reverse the value of $\epsilon_{slow,n}$ when we ``follow the green edge," but not when we ``follow the red edge."

Define a sequence of words $\rho_{slow,0,k}$ and $\rho_{slow,1,k}$ as follows, where $\omega_1 \cdot \omega_2$ simply refers to the concatenation of two words and $e$ refers to the empty word; this notation will be helpful for words which are defined with numerous subscripts and functions:
\[ \rho_{slow,0,0}=e, \qquad \rho_{slow,1,0}=e\]
\begin{align}\label{eqn:rho words} \begin{split}\rho_{slow,0,k+1} &= \begin{cases}  
\rho_{slow,0,k} & \epsilon_{slow,k}=1, \, \hat{T}_{slow}^k(x,y) \in R\\
 \sigma_{slow}(x,y,k)(A) \cdot \rho_{slow,0,k} & \epsilon_{slow,k}=1, \,\hat{T}_{slow}^k(x,y) \in G\\
\sigma_{slow}(x,y,k)(B)\cdot \rho_{slow,0,k} & \epsilon_{slow,k}=-1, \, \hat{T}_{slow}^k(x,y) \in R\\
 \rho_{slow,0,k} & \epsilon_{slow,k}=-1, \, \hat{T}_{slow}^k(x,y) \in G\\
 \end{cases} \\
\rho_{slow,1,k+1} &= 
\begin{cases} 
\sigma_{slow}(x,y,k)(B)\cdot \rho_{slow,1,k} & \epsilon_{slow,k}=1, \, \hat{T}_{slow}^k(x,y) \in R\\
\rho_{slow,1,k} & \epsilon_{slow,k}=1, \,\hat{T}_{slow}^k(x,y) \in G\\
\rho_{slow,1,k} & \epsilon_{slow,k}=-1, \, \hat{T}_{slow}^k(x,y) \in R\\
\sigma_{slow}(x,y,k)(A)\cdot \rho_{slow,1,k} & \epsilon_{slow,k}=-1, \, \hat{T}_{slow}^k(x,y) \in G.\\
 \end{cases}\end{split}\end{align}

The choice of $j=0$ or $j=1$ in $\rho_{slow,j,k}$ relates to whether these words encode the orbit of endpoints of $I_{slow}(k)$ until they reach $0^+$ or $1^-$, respectively:

\begin{lemma}\label{lemma - slow encoding lemma}The substitution $\sigma_{slow}$ maps the symbolic coding of the pair $\hat{T}_{slow}(x,y)$ to the symbolic coding of the pair $(x,y)$:
\[\Omega(x,y) = \sigma_{slow}(x,y) \left( \Omega(\hat{T}_{slow}(x,y))\right).\]

The left endpoint and right endpoint of $I_{slow}(k)$ are given by $-\|\rho_{slow,0,k}\|x$ and $-\|\rho_{slow,1,k}\|x$, respectively, where $\|\omega\|$ refers to the length of the word $\omega$.
\begin{proof}
The first claim follows from the fact that we are constructing first-return maps on intervals. When $y \in [1-x,1]$, the first-return map on an interval of length $x$ has two return times; $a$ and $a+1$, where $a=\floor{1/x}$. Furthermore, since the interval on which we are constructing the map is the interval labeled $A$, it follows that points in $A$ have their first-returns encoded either by the words $AB^{a-1}$ (when the return time is $a$) or $AB^{a}$ (when the return time is $a+1$); points begin in this interval labeled $A$, and the remainder of the orbit until returning to the interval labeled $A$ is spent in the complement labeled $B$. These words are exactly those used by $\sigma_{slow}(x,y)$ when $y \geq 1-x$. The argument is similar for $y \leq 1-x$, we simply note that we are constructing the first-return map on the interval labeled $B$, but we consider the rotation to be of length $1-x$ in the opposite orientation. This explains both composition with the substitution $\tau$ and the return times being dictated by $\floor{1/(1-x)}$.

Turning our attention to the words $\rho_{slow,i,k}$, for $k=0$ the result is immediate. We begin with $\rho_{slow,0,0}=e=\rho_{slow,1,0}$, a word of length zero, and both endpoints are of the form $-0x$. For the recursive procedure, let us consider the situation with $\epsilon_{slow,k}=1$. See \Cref{figure - two first returns in one} for both possibilities for $\epsilon_{slow,k}=1$ (we begin with a map in the forward orientation); when $\hat{T}_{slow}^k(x,y) \in R$, we use the red arrow to construct a first-return map in the interval $[0,1-x]$. In this case, the left endpoint stays the same, so $\rho_{slow,0,k+1}=\rho_{slow,0,k}$. But the new right endpoint of $1-x$ was previously encoded by $B$ in the prior system. The word $B$ therefore encodes this endpoint through one step in the previous system, at which point it would be sent to $1$. The first result of this lemma therefore gives that $\sigma_{slow}(x,y,k)(B)$ encodes the orbit of this new right endpoint until orbiting onto the previous right endpoint, so concatenation with $\rho_{slow,1,k}$ will give the word which encodes the new right endpoint until it arrives at $1$ in the starting system, rotation by $x_0$ in $I_{slow}(0)$: $\sigma_{slow}(x,y,k)(B)\cdot \rho_{slow,1,k}=\rho_{slow,1,k+1}$.

The green arrow in the same figure informs how to consider the situation $\epsilon_{slow,k}=1$, $\hat{T}_{slow}^k(x,y) \in G$. In this case the right endpoint stays the same, so $\rho_{slow,1,k+1}=\rho_{slow,1,k}$. But the left endpoint is new, having been encoded by $A$ in the previous step and being mapped to the previous left endpoint in one step: $\rho_{slow,0,k+1}=\sigma_{slow}(x,y,k)(A)\cdot \rho_{slow,0,k}$. In situations where $\epsilon_{slow,k}=-1$, the diagrams in \Cref{figure - two first returns in one} need only be reversed in orientation and considered similarly.

The words $\rho_{slow,i,k}$ therefore encode the orbits of the left and right endpoints of $I_{slow}(k)$ through their eventual returns to the endpoints of $I_{slow}(0)$, and therefore the lengths of these words are the return times necessary to return to the origin: the endpoints are given by $-\|\rho_{slow,0,k}\|x$ and $-\|\rho_{slow,1,k}\|x$.
\end{proof}
\end{lemma}

\begin{corollary}\label{corollary - quick version of slow encoding}
For $\ell=0,1$, exactly one of $\rho_{slow,\ell,k+1}=\rho_{slow,\ell,k}$, while for the other, either
$\rho_{slow,\ell,k+1} = \sigma_{slow}(x,y,k)(A)\cdot \rho_{slow,\ell,k}$ or $\rho_{slow,\ell,k+1} =\sigma_{slow}(x,y,k)(B)\cdot \rho_{slow,\ell,k}$. 
\begin{proof}
This statement follows immediately from the more precise formulation in \Cref{lemma - slow encoding lemma}. One may immediately derive this statement, however, by noting that we always construct our first return map by sharing one endpoint with the previous system, and the other endpoint was either encoded by $A$ or $B$ with a return time of one.
\end{proof}
\end{corollary}

In constructing the sequences of words $\rho_{slow,i,k}$, consecutive applications of these ``red edges with $\hat{a}_{k+1}=1$" correspond to consecutive terms where $\epsilon_{slow,k}=\epsilon_{slow,k+1}$. In \Cref{eqn:rho words}, this condition exactly describes when the words $\sigma_{slow}(x,y,n)(B)$ are considered. However, the individual substitutions being generated map $B \mapsto B$ in this special situation.

In other words, for $(x,y)\in PG_{i,n}$, to determine the action of $\hat{T}_{fast}$ as iterations of $\hat{T}_{slow}$, the map $\hat{T}_{slow}$ would successively act as ``red edge, $a=1$" a total of $i$ consecutive times, but for each $j=1,\ldots,i$ we will have
\[ \sigma_{slow}(x,y,j)(B) = B\]
because the first $i$ substitutions applied will all map $B \mapsto B$. Similarly, if $(x,y) \in PR_{i,n}$, then as we have exactly $n-1$ consecutive such applications of ``red edge, $a=1$", for all $j=1,\ldots,n-1$ we also have
\[ \sigma_{slow}(x,y,j)(B) =B.\]
After these edges, we follow a single application of $\hat{T}_{slow}$ either with a ``green edge," or a ``red edge with $\hat{a}_{k+1} \geq 2$," and we have determined how substitutions help encode a single such instance.

Altogether, we use these observations to define the sequences $\rho_{fast,0,n}$ and $\rho_{fast,1,n}$ which encode the orbits of the left and right endpoints of $I_{fast}(n)$. Let the $\epsilon_{fast,n}=\pm 1 $ similarly to the $\epsilon_{slow,n}$; these values track the number of times we have reversed orientation due to $\hat{T}_{fast}^i(x,y)$ belonging to some $PG_{i,n}$. Specifically:
\[ \epsilon_{fast,0}=1, \qquad \epsilon_{fast,n+1} = \begin{cases} \epsilon_{fast,n} & \hat{T}_{fast}^n(x,y) \in PR_{i,n} \\ -\epsilon_{fast,n} & \hat{T}_{fast}^n(x,y) \in PG_{i,n}.\end{cases}\]
Begin again with $\rho_{fast,0,0}=e$ and $\rho_{fast,1,0}=e$, and then
\begin{align}
\label{eqn:rho fast words}
\begin{split}
    \rho_{fast,0,k+1} &= \begin{cases}
    \rho_{fast,0,k} & \epsilon_{fast,k}=1, \, \hat{T}_{fast}^k(x,y) \in PR_{i,n}\\
    \sigma_{fast}(x,y,k)(B^i A)\cdot \rho_{fast,0,k} & \epsilon_{fast,k}=1, \, \hat{T}_{fast}^k(x,y) \in PG_{i,n}\\
    \sigma_{fast}(x,y,k)(B^{n})\cdot \rho_{fast,0,k}& \epsilon_{fast,k}=-1, \, \hat{T}_{fast}^k(x,y) \in PR_{i,n}\\
    \sigma_{fast}(x,y,k)(B^i)\cdot \rho_{fast,0,k}& \epsilon_{fast,k}=-1, \, \hat{T}_{fast}^k(x,y) \in PG_{i,n}
    \end{cases}\\
    \rho_{fast,1,k+1} &= \begin{cases}
    \sigma_{fast}(x,y,k)(B^n)\cdot\rho_{fast,1,k} & \epsilon_{fast,k}=1, \, \hat{T}_{fast}^k(x,y) \in PR_{i,n}\\
    \sigma_{fast}(x,y,k)(B^i)\cdot\rho_{fast,1,k} & \epsilon_{fast,k}=1, \, \hat{T}_{fast}^k(x,y) \in PG_{i,n}\\
    \rho_{fast,1,k} & \epsilon_{fast,k}=-1, \, \hat{T}_{fast}^k(x,y) \in PR_{i,n}\\
    \sigma_{fast}(x,y,k)(B^i A)\cdot \rho_{fast,1,k} & \epsilon_{fast,k}=-1, \, \hat{T}_{fast}^k(x,y) \in PG_{i,n}.\\
    \end{cases}
\end{split}
\end{align}

\begin{lemma}\label{lemma - fast encoding lemma}The substitution $\sigma_{fast}$ maps the symbolic coding of the pair $\hat{T}_{fast}(x,y)$ to the symbolic coding of the pair $(x,y)$:
\[\Omega(x,y) = \sigma_{fast}(x,y) \left( \Omega(\hat{T}_{fast}(x,y))\right).\]
The words $\sigma_{fast}(x,y,n)(A)$ and $\sigma_{fast}(x,y,n)(B)$ give encodings of the endpoints of $I_{fast}(n)$ through their length.
The left endpoint and right endpoint of $I_{fast}(n)$ are given by $-\|\rho_{fast,0,n}\|x$ and $-\|\rho_{fast,1,n}\|x$, respectively.
\begin{proof}
The first statement is immediate in light of \Cref{lemma - slow encoding lemma} and \Cref{equation - fast sigma}. The information about the $\rho_{fast,i,k}$ follows from a case-by-case analysis, determining the encoding of the left/right endpoints of the interval on which we are going to construct our first-return map until it reaches the previous left/right endpoints.

For example, consider the situation of $\epsilon_{fast,k}=1$, so we are rotating from left to right. Suppose that $\hat{T}_{fast}^k(x,y) \in PG_{i,n}$. Equivalently, if we let $(x',y')=\hat{T}^k_{fast}(x,y)$, then we will be constructing our first-return map on the (normalized) interval $[1-(i+1)x',1-ix']$. The left endpoint takes $(i+1)$ steps to return to the origin. In the interim it will have $i$ consecutive values less than $1-x'$, then a single value equal to (the right-sided version of) $1-x'$. So this portion of the orbit is encoded by $B^i A$, after which we have arrived at the (right-sided version of the) origin, so the orbit continues with $\rho_{fast,0,k}$. Overall:
\[\rho_{fast,0,k+1}= \sigma_{fast}(x,y,k)(B^iA)\cdot \rho_{fast,0,k}.\]

In contrast, the right endpoint of $1-ix'$ will orbit to the (left-sided version of the) origin after $i$ steps, having its orbit encoded by $B^i$:
\[\rho_{fast,1,k+1}=\sigma_{fast}(x,y,k)(B^i)\cdot \rho_{fast,1,k}.\]
Other cases are handled similarly.
\end{proof}
\end{lemma}

\begin{corollary}\label{corollary - quick version of fast encoding}
Between $\ell=0,1$, if $\hat{T}_{fast}(x,y,k) \in PR_{i,n}$, exactly one of the words $\rho_{fast,\ell,k+1}$ is given by $\rho_{fast,\ell,k}$, while the other is given by $\sigma_{fast}(x,y,k)(B^n)\cdot\rho_{fast,\ell,k}$. 

If $\hat{T}_{fast}(x,y,k) \in PG_{i,n}$, then one of the words $\rho_{fast,\ell,k+1}$ is given by $\sigma_{fast}(x,y,k)(B^i\cdot \rho_{fast,\ell,k}$ while the other is given by $\sigma_{fast}(x,y,k)(B^iA)\cdot \rho_{fast,\ell,k}$.
\begin{proof}
Just as \Cref{corollary - quick version of slow encoding} followed immediately from \Cref{lemma - slow encoding lemma}, this result follows immediately from \Cref{lemma - fast encoding lemma}.
\end{proof}
\end{corollary}

We wish to study the growth rate of the slow and fast approximating sequences, so by \Cref{lemma - slow encoding lemma}, \Cref{lemma - fast encoding lemma}, we equivalently wish to study the growth rates of concatenations of various words generated by the substitutions $\sigma_{slow}(x,y,k)$ and $\sigma_{fast}(x,y,k)$. Our first step is to show that up to a linear factor, the growth rate is determined by the new term being concatenated, and not the previous terms generated. Define
\begin{align*}
N_{slow}(x,y,k)&=\max\left\{ \left\|\rho_{slow,0,k}\right\|,\left\|\rho_{slow,1,k}\right\|\right\},\\
MAX_{slow}(x,y,k)&=\max\left\{ \left\|\sigma_{slow}(x,y,k)(A)\right\|,\left\|\sigma_{slow}(x,y,k)(B)\right\|\right\},\\
min_{slow}(x,y,k)&=\min\left\{ \left\|\sigma_{slow}(x,y,k)(A)\right\|,\left\|\sigma_{slow}(x,y,k)(B)\right\|\right\}\\
\end{align*}
and then also define
\begin{align*}
N_{fast}(x,y,k)&=\max\left\{ \left\|\rho_{fast,0,k}\right\|,\left\|\rho_{fast,1,k}\right\|\right\},\\
MAX_{fast}(x,y,k)&=\max\left\{ \left\|\sigma_{fast}(x,y,k)(A)\right\|,\left\|\sigma_{fast}(x,y,k)(B)\right\|\right\},\\
min_{fast}(x,y,k)&=\min\left\{ \left\|\sigma_{fast}(x,y,k)(A)\right\|,\left\|\sigma_{fast}(x,y,k)(B)\right\|\right\}.\\
\end{align*}

\begin{lemma}
\label{lemma:word growth estimates}
For any $x,y$, we have
\[ min_{slow}(x,y,k) \leq N_{slow}(x,y,k+1) \leq (k+1) \cdot MAX_{slow}(x,y,k+1),\]
\[ min_{fast}(x,y,k) \leq N_{fast}(x,y,k+1) \leq (k+1) \cdot  MAX_{fast}(x,y,k+1).\]
\begin{proof}
The lower bounds are immediate; in light of \Cref{corollary - quick version of slow encoding}, at least one of the words $\rho_{slow,0,k+1}$, $\rho_{slow,1,k+1}$ will involve concatenation with $\sigma_{slow}(x,y,k)(A)$ or $\sigma_{slow}(x,y,k)(B)$ so whichever of the two $\rho_{slow}$ is larger, $N_{slow}(x,y,k+1)$ will be at least as large as the smaller of those two words. The result is similar for the lower bound on the $N_{fast}$; by \Cref{corollary - quick version of fast encoding} some $\sigma_{fast}(x,y,k)(A)$ or $\sigma_{fast}(x,y,k)(B)$ will always be concatenated to one of the previous $\rho_{fast}$.

For the upper bound on $N_{slow}$, we prove by induction. Note that for $k=0$ we are considering $N_{slow}(x,y,1)$, which is the length of either $\rho_{slow,0,1}$ or $\rho_{slow,1,1}$, whichever is larger. But those words refer to the encoding of the endpoints of the first interval on which we construct a first return map; those points are either $0,1-x$ or $1-x,1$. In either case, one is encoded with a word of length zero, and the other with a word of length one. That is: $N_{slow}(x,y,1)=1$ always: our base case is shown.

So assume for some $k-1$ that the upper inequality holds: $N_{slow}(x,y,k) \leq k MAX_{slow}(x,y,k)$. Then in light of \Cref{corollary - quick version of slow encoding}, the longer of the two words $\rho_{slow,0,k+1}$, $\rho_{slow,1,k+1}$ is no longer than the concatenation of either $\sigma_{slow}(x,y,k)(A)$ or $\sigma_{slow}(x,y,k)(B)$ with whichever of $\rho_{slow,0,k}$, $\rho_{slow,1,k}$ was longer:
\begin{align*}
    N_{slow}(x,y,k+1) &\leq MAX_{slow}(x,y,k)+N_{slow}(x,y,k)\\
    &\leq MAX_{slow}(x,y,k)+ k MAX_{slow}(x,y,k)\\
    &\leq (k+1)MAX_{slow}(x,y,k)\\
    &\leq (k+1)MAX_{slow}(x,y,k+1)
\end{align*}
The last line is included only to make the inequality appear similar to the inequalities for $N_{fast}$; the upper bound on the $N_{fast}$ follows from a similar argument. For $k=0$, we have $N_{fast}(x,y,1)$ representing the length of the encoding of some point in the partition $\{0,1-ax,1-(a-1)x,\ldots,1-x,1\}$, so the length of that word is no larger than $a=\floor{1/x}$. On the other hand, $1 \cdot MAX_{fast}(x,y,1)$ is the length of some $\sigma(A)$ or $\sigma(B)$, which are seen to be of length at least $a$. By \Cref{corollary - quick version of fast encoding}, $N_{fast}(x,y,k+1)$ is no larger than concatenating the longest possible word to a word of length $N_{fast}(x,y,k+1)$. But the word we concatenate is of the form $\sigma_{fast}(x,y,k)$ of some factor of the single substitution generated by $(x_k,y_k)$ (see \Cref{eqn:rho fast words}). Therefore the word we concatenate is no larger than $MAX_{fast}(x,y,k+1)$ (the relevant substitution for this term ends with the substitution generated by $(x_k,y_k)$, explaining the appearance of one larger index than naturally occurred in the ``slow" situation, where single letters were always fed into substitutions):

\[N_{fast}(x,y,k+1) \leq MAX_{fast}(x,y,k+1)+N_{fast}(x,y,k).\]
The upper bound now follows inductively as it did in the case of $N_{slow}$.
\end{proof}
\end{lemma}
We remark that the upper bound in both cases is likely far from optimal, but will be sufficient for our desired claims. By definition, the slow and fast approximating sequences are given by $N_{slow}(x,y,k)$ and $N_{fast}(x,y,k)$, and the orbit of $y$ is encoded through these lengths by the relevant words $\rho$. We may now state the objective of this section:

\begin{corollary}\label{cor:approximating sequence growth rates}
There is a generic growth rate of the fast-approximating times $N_{fast}(x,y,k)$: there exists some $C>1$ so that for almost every choice of $x,y$, we have
\[ \lim_{k \rightarrow \infty} \frac{1}{k} \log\left( N_{fast}(x,y,k) \right) = C.\]
However, the slow-approximating times $N_{slow}(x,y,k)$ have no such generic growth rate: for any $C>1$, for almost all $x,y$ we have either
\[ \liminf_{k \rightarrow \infty} \frac{N_{slow}(x,y,k)}{C^k} = 0 \qquad \textit{or} \qquad \limsup_{k \rightarrow \infty} \frac{N_{slow}(x,y,k)}{C^k} = \infty.\]
\begin{proof}
By the multiplicative ergodic theorem, \Cref{theorem - T_fast is ergodic}, and \Cref{theorem - lambda integrable}, there is some $C>1$ so that almost surely both
\[ \lim_{k \rightarrow \infty} \frac{1}{k}\log \| \sigma_{fast}(x,y,k)(A)\| = C = \lim_{k \rightarrow \infty} \frac{1}{k}\log \| \sigma_{fast}(x,y,k)(B)\|,\]
from which we conclude that
\[ \lim_{k \rightarrow \infty} \frac{1}{k} \log min_{fast}(x,y,k) = C = \lim_{k \rightarrow \infty}\frac{1}{k} \log MAX_{fast}(x,y,k).\]
We then recall the estimates of \Cref{lemma:word growth estimates}, and the first result is shown.

The second claim is then a standard result in theory of transformations which preserve an infinite measure, see e.g \cite[Theorem 2.4.2]{aaronson1997introduction}. 
\end{proof}
\end{corollary}

It is proved in \cite{IN} that $C=\pi^2/(12 \log 2)$, the almost-sure base of the growth rate of the denominators $q_n$ in the convergents of the standard continued fraction representation of $x$. Recall that what we call the fast approximating sequence for $(x,y)$ is what in that reference is called the canonical approximating sequence for $(x,1-y)$, but results for generic rates of growth are equivalent.
\section{Relationships to Continued Fractions}\label{section - relation to continued fractions}

In this section we do not first pick a $y$ and follow the algorithm to see when to follow the green edge and when to follow the red edge. Rather, we follow certain methods for choosing whether to construct the first-return map on $[0,1-x]$, or whether to reverse orientation and construct the first-return map on $[1-x,1]$. Different systems for making these choices will be shown to correspond to different continued fraction systems. The unique point of intersection of these nested intervals therefore provides a special point whose orbit is encoded by substitutions that are related to these continued fraction systems. To avoid unnecessary ambiguity we assume that  $x\in[0,1]\setminus \mathbb{Q}$.

Suppose the regular continued fraction expansion of $x$ is given by
\[x = \cfrac{1}{a_1+\cfrac{1}{a_2+\cfrac{1}{a_3+\ddots}}}=[a_1,a_2,a_3,\ldots].\]
Then we see that
\begin{equation}\label{eq:subtraction} 1-x = \begin{cases} [a_2+1,a_3,\ldots] & (a_1=1), \\ [1,a_1-1,a_2,\ldots] & (a_1 \neq 1). \end{cases}\end{equation}

Accordingly, since certainly $x=1-(1-x)$, we find that ``following the red edge" exactly corresponds to performing singularization or insertion on the continued fraction expansion of $x$; singularization when some $a_n=1$, insertion when $a_n \neq 1$.
To be more precise, a singularization is obtained by using the following formula for $a,b\in\mathbb{N}_{>0}$ and $\varepsilon=\pm 1$
\begin{equation}\label{eq:sing}
    a+\frac{\varepsilon}{\displaystyle1+\frac{1}{b+\xi}} =a+\varepsilon+\frac{-\varepsilon}{b+1+\xi}.
\end{equation}
  
Taking $a=1$ and $\varepsilon=-1$ we find the continued fraction of $1-x$. We see that $T^2(x)=T(1-x)$ so inducing on the interval $[x,1]$ will `go faster' in the sense that the associated return times are larger. This corresponds to the fact that we are constructing the return map on the smaller interval ($a_1=1$ gives $x>\frac{1}{2}$).
When $a_1>1$ we get insertion. The formula for insertion where $a,b\in\mathbb{N}_{>0}$ with $b\geq 2$ and $\xi\in [0,1]$ is given by
\begin{equation}\label{eq:ins}
a+\frac{1}{b+\xi}=a+1+\frac{-1}{\displaystyle 1+\frac{1}{\displaystyle b-1+\xi}}.
\end{equation}
Taking $a=0$ we get the continued fraction of $x$ on the left hand side and $1-(1-x)$ on the right hand side when $a_1\neq 1$. Here we see that $T(x)=T^2(1-x)$ which means we `went slower' in the sense of smaller return times: we constructed the return map on the larger interval.
By using singularizations and insertions one can find all semi-regular continued fractions of a  number \cite{DK,K}. That is, in case of irrational numbers, all expansions with numerators $\varepsilon_n\in\{-1,1\}$, digits $\hat{a}_n\in\mathbb{N}$ and for all $n$, $\hat{a}_n+\varepsilon_n\geq 1$, and infinitely often $\hat{a}_n+\varepsilon_n\geq 2$. Each such representation of $x$ corresponds to a certain choice of $y$; the $y \in [0,1]$ for which our algorithm would transform the standard continued fraction expansion of $x$ into the desired form: $\hat{a}_n$ in the modified continued fraction are exactly the return times $\hat{a}_{n}$ from \Cref{figure - graph presentation}. Conversely, all such $y$ generate a semi-regular continued fraction representation of $x$. We therefore have a canonical correspondence between all semi-regular continued fraction representations of some $x$ (with numerators $\pm 1$) and the interval $[0,1]$ (with left/right versions of those $y \in x \mathbb{Z}$):
\begin{theorem}\label{theorem:all CF expansions}
    Every $y\in[0,1]$ generates exactly one semi-regular continued fraction of $x$ and for every semi-regular continued fraction of $x$ there is exactly one $y\in[0,1]$ that generates it.
\end{theorem}
\begin{proof}
    The proof is clear from the previous discussion. 
\end{proof}

Another way of generating semi-regular continued fractions is to pick the numerators $\varepsilon_n$ at random. This is done, for example, in \cite{KKV}. Note that heuristic arguments about `random' choices of $\varepsilon_n$ will not necessarily agree with almost-sure results for $\hat{T}_{fast}$, as singularization/insertion are performed by $\hat{T}_{slow}$ for any $y \leq 1-x$. When the first partial quotient of $x$ is very large, the probability of singularization/insertion under $\hat{T}_{slow}$ is correspondingly very large.

Singularizations and insertions are local operations that can be applied to any semi-regular continued fraction of $x$. If 
\[
x=[1/d_1,\varepsilon_1/d_2,\ldots,\varepsilon_{n-1}/d_n, \varepsilon_n/1,1/d_{n+2},\ldots]
\]
then applying a singularization at $d_{n+1}=1$ gives
\[
x=[1/d_1,\varepsilon_1/d_2,\ldots,\varepsilon_{n-1}/(d_n+\varepsilon_n), -1/d_{n+2},\ldots].
\]
On the other hand when
\[
x=[1/d_1,\varepsilon_1/d_2,\ldots,\varepsilon_{n-1}/d_n, 1/d_{n+1},\varepsilon_{n+1}/d_{n+2},\ldots]
\]
then, when $d_{n+1}\geq 2$, insertion after $d_n$ gives us
\[
x=[1/d_1,\varepsilon_1/d_2,\ldots,\varepsilon_{n-1}/(d_n+1),-1/1, 1/(d_{n+1}-1),\varepsilon_{n+1}/d_{n+2},\ldots].
\]
Furthermore, these operations affect the convergents in the following way.
Write $\frac{p_n(x)}{q_n(x)}=[1/d_1,\varepsilon_1/d_2,\ldots,\varepsilon_{n-1}/d_n]$. Then performing a singularization at place $n$ will result in deleting  the element $\frac{p_n}{q_n}$ in the list of convergents, i.e. $\frac{p_1}{q_1},\frac{p_2}{q_2},\ldots \frac{p_{n-1}}{q_{n-1}},\frac{p_{n+1}}{q_{n+1}},\ldots$.
On the other hand insertion will add a convergent $\frac{p_1}{q_1},\frac{p_2}{q_2},\ldots \frac{p_{n-1}}{q_{n-1}},\frac{p_{n-1}+p_{n}}{q_{n-1}+q_n},\frac{p_{n}}{q_{n}},\ldots$. The numerators $p_n$ and denominators $q_n$ for the altered, semi-regular expansion of $x$, satisfy the following convergent relations

\begin{equation}\label{convergents}
	    \begin{split}
		    p_{-1}=1, \quad p_0=0, \quad p_n=d_np_{n-1}+\varepsilon_{n-1} p_{n-2},\\
	        q_{-1}=0, \quad q_0=1, \quad q_n=d_nq_{n-1}+\varepsilon_{n-1} q_{n-2}.
		\end{split}
	\end{equation}
Deciding when to do an insertion or an insertion is dictated by $y\in[0,1]$. 
With this in mind we define $q_{n,y}(x)$ to be the denominators of the sequence of convergents of $x$ where the places of insertion and singularization are described by the orbit of $(x,y)$ when iterating over $\hat{T}_{slow}$. Furthermore, we can look at those $n_k$ such that correspond to the accelerated system $T_{fast}$; let this be $\overline{q_{k,y}}(x)=q_{n_k,y}(x)$. Observe that as $\hat{T}_{fast}$ is defined as iterating $\hat{T}_{slow}$ until either $y \geq (1-x)$ (which corresponded to regular progression in the convergents of $x$) or $y \leq 1-x$ and $x \geq 1/2$ (which corresponded to singularization), the $\overline{q_{k,y}}(x)$ correspond to those $q_{n,y}(x)$ which are \textit{not} due to insertion. We are interested in how fast $q_{n,y}(x)$ and $\overline{q_{n,y}}(x)$ typically grow. To this end, let us look  at the inverse branches of $\hat{T}_{slow}$ as M\"obius transformations acting on the matrix  with the convergents. To be more precise, 
\[
\left[\begin{array}{cc}
a & b \\
c & d 
\end{array}\right](x):=\frac{ax+b}{cx+d}.
\]
Then the inverse branches of $\hat{T}_{slow}$ are given by
\[
A_{x,y}:=\begin{cases}
    \left[\begin{array}{cc}
0 & 1 \\
1 & \hat{a}_1(x,y)
\end{array}\right] & \text{when } y\in [1-x,1]\\[1em]
\left[\begin{array}{cc}
1 & \hat{a}_1(x,y) -1 \\
1 & \hat{a}_1(x,y) 
\end{array}\right] &  \text{when } y\in [0,1-x].
\end{cases}
\]
Let \[
M_{0,x,y}=
\left[\begin{array}{cc}
1 & 0 \\
0 & 1
\end{array}\right]
=
\left[\begin{array}{cc}
p_{-1} & p_0\\
q_{-1} & q_0 
\end{array}\right]
\]
and $M_{n+1,x,y}=M_{n,x,y}A_{T_{slow}^n(x,y)}$ for $n\geq 1$.
Now when  $y_n\in [1-x_n,1] $ we find 
\[
M_{n+1,x,y}=
\left[\begin{array}{cc}
p_{n-1} & p_n \\
q_{n-1} & q_n
\end{array}\right]
\left[\begin{array}{cc}
0 & 1 \\
1 & \hat{a}_{n+1}(x,y)
\end{array}\right]
=
\left[\begin{array}{cc}
p_n & p_{n+1} \\
q_n & q_{n+1}
\end{array}\right]
\]
just like in the regular case (since we applied $T(x)$). Note that the index of the convergents may not be the same as for the convergents of it's regular continued fraction. 

Now when  $y_n\in [0,1-x_n] $, 
if $\hat{a}_{n+1}(x,y)=1$ we find 
\[
M_{n+1,x,y}=
\left[\begin{array}{cc}
p_{n-1} & p_n \\
q_{n-1} & q_n
\end{array}\right]
\left[\begin{array}{cc}
1 & 0 \\
1 & 1
\end{array}\right]
=
\left[\begin{array}{cc}
p_{n-1}+p_n & p_n \\
q_{n-1}+q_n & q_n
\end{array}\right].
\]
We see that we did not update the second column but instead replaced the first by the mediant (corresponding to adding the mediant in the list of convergents). To illustrate what would happen in case of a singularization, for ease of notation assume that the matrix $M_{n,x,y}$ has the convergents of the regular continued fraction as columns for $n-1$ and $n$ and we will be using the $n+1$'th digit of $x$ in the next step.
When $a_{n+1}(x)=1$ then $a_{n+1}(1-x)=a_{n+2}(x)+1$. We find
\[
\left[\begin{array}{cc}
p_{n-1} & p_n \\
q_{n-1} & q_n
\end{array}\right]
\left[\begin{array}{cc}
1 & a_{n+2}(x) \\
1 & a_{n+2}(x)+1
\end{array}\right]
=
\left[\begin{array}{cc}
p_{n-1}+p_n & a_{n+2}(x)p_{n-1}+(a_{n+2}(x)+1)p_n \\
q_{n-1}+q_n & a_{n+2}(x)q_{n-1}+(a_{n+2}(x)+1)q_n
\end{array}\right]=
\left[\begin{array}{cc}
p_{n+1} & p_{n+2} \\
q_{n+1} & q_{n+2}
\end{array}\right].
\]
Here we see that we skipped a convergent because of a singularization. (Here the index might have shifted as well.) Since we now know how to relate the inverse branches to the sequence of convergents $q_{n,y}(x)$ and, in the same way, $\overline{q_{n,y}}(x)$ we are in a position to proving the following:

\begin{theorem}\label{theorem:generic denominator growth}
There is a generic growth rate of  $\overline{q_{n,y}}(x)$: there exists some $C>1$ so that for almost every choice of $x,y$, we have
\[ \lim_{n \rightarrow \infty} \frac{1}{n} \log\left( \overline{q_{n,y}}(x) \right) = C.\]
However, for $q_{n,y}(x)$ have no such generic growth rate: for any $C>1$, for almost all $x,y$ we have 
\[ 
\liminf_{n \rightarrow \infty} \frac{q_{n,y}(x) }{C^n} = 0. 
\]

\end{theorem}

\begin{proof}
For the fast map, for $(x,y)\in PG_{i,n}$, let $n=a_1(x)=\floor{1/x}$. In $PG_{i,n}$ we first do $i$ insertions and then apply the Gauss map, so we find the following matrix multiplication
\[
\overline{A}_{x,y}:=
\left[\begin{array}{cc}
1 & 0 \\
1 & 1
\end{array}\right]^i
\left[\begin{array}{cc}
0 & 1 \\
1 & n-i
\end{array}\right]=
\left[\begin{array}{cc}
0 & 1 \\
1 & n
\end{array}\right]
\]
which is the same as applying the Gauss map directly. In this case we find the eigenvalues
\[
\frac{n\pm \sqrt{n^2+4}}{2}.
\]
When $(x,y)\in  PR_{i,n}$, then we have $n-1$ insertions followed by a singularization. Note that $a_2(x)=i$. We find the following matrix multiplication
\[
\overline{A}_{x,y}:=
\left[\begin{array}{cc}
1 & 0 \\
1 & 1
\end{array}\right]^{n-1}
\left[\begin{array}{cc}
1 & i \\
1 & i+1
\end{array}\right]=
\left[\begin{array}{cc}
1 & i \\
n & ni+1
\end{array}\right].
\]
The eigenvalues are
\[
\frac{ (ni+2) \pm \sqrt{(ni+2)^2 -4}}{2}.
\]
We see that we are almost in the same situation as for $M_{fast}(x,y)$, see \eqref{eigMfast}. Here, though, the principal eigenvalues are smaller in case $(x,y)\in PR_{i,n}$. This gives us the integrability of $\log^+\|\overline{A}_{x,y} \|$. The second eigenvalue is larger than the case of $M_{fast}(x,y)$ and therefore the norm of the inverse is smaller. This gives us integrability of  $\log^+\|\overline{A}_{x,y}^{\, -1} \|$.
The integrability of the logarithm of these norms, the multiplicative ergodic theorem, and the relation of the matrices with the sequences $\overline{q_{n,y}}(x)$ gives us the desired result. The second statement in the theorem follows from the fact that $T_{slow}$ has an infinite invariant measure. Informally put, $\hat{T}_{slow}$ will perform `too many insertions of mediants' to permit any generic exponential growth rate.
\end{proof}

\subsection{Standard Continued Fractions - \emph{Always Follow the Green Edge}}

When we construct the first-return map on $[1-x,1]$, we get a rotation by $T(x)$, the Gauss map. It follows that if $y_0$ is chosen so that $y_n$ is \emph{always} in the interval $[1-x_n,1]$, then $x_n = T^n(x_0)$ for every $n$.

Our procedure then amounts to ``within the interval whose length is the rotation amount $x$, we construct the first-return map on the interval which is the first $T(x)$ proportion, and on this interval we reverse orientation." Let us let $I_n$ be the interval in the original interval $[0,1]$ on which we are constructing a first-return map after $n$ consecutive iterations: the first-return map on $I_n$ is therefore isomorphic to rotation by $T^n(x)$.

So beginning with the interval 
\[I_1=[1-x,1]\] with forward orientation, our next interval would be of length $xT(x)$ and sharing the left endpoint with the previous interval: 
\[I_2=[1-x,1-x+xT(x)].\] In our first-return map our orientation has been reversed, so our third interval will therefore be of length $xT(x)T^2(x)$, but sharing the \emph{right} endpoint of the previous interval: 
\[I_3=[1-x+xT(x)-xT(x)T^2(x),1-x+xT(x)].\] Next, the interval will be of length $xT(x)T^2(x)T^3(x)$, but sharing the \emph{left} endpoint, for a result of 
\[I_4=[1-x+xT(x)-xT(x)T^2(x),1-x+xT(x)-xT(x)T^2(x)+xT(x)T^2(x)T^3(x)].\]
Continuing in this fashion, we see inductively that $I_n$ is of length $xT(x)\cdots T^{n-1}(x)$, and that if we let $y_g$ be the unique point of intersection, then
\[ y_g=\bigcap_{i=1}^{\infty} I_n = 1-x+xT(x)-xT(x)T^2(x)+xT(x)T^2(x)T^3(x)-\cdots\]
So $y_g$ is the unique point in $[0,1]$ whose orbit is exactly encoded by 
\[\omega = \lim_{n \rightarrow \infty} \sigma_{[1,n)}(A),\]
where each 
\[
\sigma_i: \left\{ 
\begin{array}{l}
A \mapsto AB^{a_i} \\ B \mapsto AB^{a_i-1}
\end{array}\right.,
\]
and the $a_i$ are given by the standard continued fraction representation $x = [0;a_1,a_2,\ldots]$. Also note that this is the unique $y$ such that $\hat{T}^n_{slow}(x,y)\in G$ for all $n\in \mathbb{N}$.

In the special case where $x$ has a purely periodic continued fraction expansion with period one, this point and the substitutions are fairly easy to compute. Suppose that
\[ x = \cfrac{1}{a+\cfrac{1}{a+\cfrac{1}{a+\ddots}}} = [a,a,a,\ldots] = \frac{\sqrt{a^2+4}-a}{2}.\]
The point $y_g$ can then be found to be given by
\[y_g = 1-x+x^2-x^3+x^4+\cdots = \frac{1}{1+x}\]
Then the only substitution is
\[\sigma: \left\{ \begin{array}{l}A \mapsto AB^a \\ B \mapsto AB^{a-1}\end{array}\right.,\]
and the orbit of $y_g$ is encoded by 
\[
\lim_{n \rightarrow \infty} \sigma^n(A).
\]
Furthermore, for almost every $x$ we have
\[
\lim_{n\to \infty} \frac{1}{n}\log (q_{n,y_g} )=\frac{\pi^2}{12\log(2)}=\frac{1}{2}h_\mu(T)
\]
where $h_\mu(T)$ is the measure theoretic entropy with respect to the invariant measure $\mu$.
This is a classical result and can be found for example in \cite{DK2,DKa} Note that we also have $\hat{T}_{fast}^n(x,y_g)=\hat{T}_{slow}^n(x,y_g)$ so that $q_{n,y_g}(x)=\overline{q_{n,y_g}}(x)$ for all  $x\in[0,1]\backslash \mathbb{Q}$.

\subsection{Other continued fractions}
For other choices of $y$ we will get different semi-regular continued fractions for $x$. By choosing $y$ in a specific manner according to the desired sequence of `red edges' and `green edges,' we specify a sequence of desired insertions/singularizations, and therefore we can pick $y$ to produce certain well-studied continued fraction expansions. We briefly discuss several examples.
\subsubsection{The backward continued fractions -\emph{Always follow the red edge!} }

When always following the red edge we find that the first return map on $I_n$ is isomorphic to the rotation by $\tilde{T}^n(x)$ where $\tilde{T}(x)=1-T(x)$. The map $\tilde{T}$ is isomorphic to the map that generates the backward continued fractions introduced in 1957 by R\'enyi \cite{R}, for which the numerators are always $-1$. Because we never reverse orientation in this scheme, our intervals $I_n$ will always share their left endpoint with $I_{n-1}$. The point on intersection is therefore the left endpoint of the original interval: we are producing substitutions which encode the orbit of the point $y=0$. Another way to determine $y$ is by looking at the map $\hat{T}_{slow}$. When for every $n\in\mathbb{N}$ we have that $\hat{T}^n_{slow}(x,y)\in[0,1]\times[0,1-x]$ then we always follow the red edge. For $y=0$ we have $\hat{T}^n_{slow}(x,y)= T(1-T(1-T(1-\cdots T(1-x))),0)$. Likewise, we find for $y=0$ we have $\hat{T}^n_{fast}(x,y)=(T^{2n}(x),0)$ for every $n\in\mathbb{N}$. The \textit{fast} approximating sequence is therefore the sequence $\overline{q_{n,y}}(x)=q_{2n}(x)$, the denominators of the even-index convergents of the standard continued fraction expansion of $x$. Equivalently, $I_{slow}(n)=[0,\|q_{2n}x\|]$. Note that for $y=1$ we have $\hat{T}_{slow}(x,1)=(T(x),0)$. Therefore, we find $\overline{q_{n,1}}(x)=q_{2n+1}(x)$, all odd-index convergents, and $I_{slow}(n)=[1-\|q_{2n+1}x\|,1]$.

\subsubsection{\texorpdfstring{$\alpha$}{}-Continued Fractions - \emph{Follow the Green Edge Unless \texorpdfstring{$x_n>\alpha$}{}!}}
Let $\alpha\in[0,1]$. In \Cref{figure - graph presentation}, follow the green arrow if $x_n\leq \alpha$, otherwise follow the red arrow. Correspondingly, we perform singularization/insertion exactly when $x_n > \alpha$, as in \cite{DHKM}. This results in the $\alpha$-continued fraction of $x$, introduced in \cite{N}. Note that when $\alpha=1$ we always follow the green edge and find the regular continued fraction and in case $\alpha=0$ we always follow the red edge and we find the backward continued fractions. Write $y_\alpha$, which also depends on $x$, for the corresponding $y$ value. This is the unique $y$ such that $\hat{T}_{slow}^n(x,y)\in [0,\alpha]\times [1-x,1]\cup (\alpha,1]\times [0,1-x]$. Furthermore, for almost all $x\in[0,1]$ we have 
\[
\lim_{n\to \infty} \frac{1}{n}\log (q_{n,y_\alpha} )=\frac{1}{2}h_{\mu_\alpha}(T_\alpha)
\]
where $T_\alpha$ is Nakada's $\alpha$-continued fraction map and $\mu_\alpha$ it's invariant measure that is absolutely continuous with respect to the Lebesgue measure. The map $h(\alpha)=h_{\mu_\alpha}(T_\alpha)$ is very well studied, see for example \cite{CT12,CT13,KSS,LM,NN08}. 

\subsubsection{Nearest Integer CF - \emph{Follow the Edge with the Shorter Interval!}}
For $\alpha=1/2$, the above reduces to ``follow the unique arrow for which $x_{n+1}<1/2$." This is equivalent to ``pick the substitutions whose word length grows the fastest" or ``construct the first-return map on whichever interval is smaller." The corresponding continued fraction representation of $x$ is the unique such expansion for which no partial quotient is ever one. We find the unique $y$ such that $\hat{T}_{slow}^n(x,y)\in [0,\frac{1}{2}]\times [1-x,1]\cup (\frac{1}{2},1]\times [0,1-x]$. For almost all $x\in[0,1]$ we find 
\[
\lim_{n\to \infty} \frac{1}{n}\log (q_{n,y_\frac{1}{2}} )=\frac{1}{2}h_{\mu_\frac{1}{2}}(T_\frac{1}{2})=\frac{\pi^2}{12\log\left(\frac{\sqrt{5}-1}{2}\right)},
\]
 see for example \cite{CT13}.

\subsubsection{A natural counterpart to \texorpdfstring{$\alpha$}{}-Continued Fractions - \emph{Follow the Green Edge Unless \texorpdfstring{$x<\alpha$}{}!} }
Now instead of following the green arrow in  \Cref{figure - graph presentation} when $x\leq \alpha $ we follow  the \textit{red arrow} in that case. These continued fractions are studied in \cite{KLMM}. The digits of the resulting continued fraction expansions are bounded from above. The corresponding $y$ is the unique $y$ such that $\hat{T}_{slow}^n(x,y)\in [0,\alpha]\times [0,x]\cup (\alpha,1]\times [1-x,1]$. Such continued fraction expansion generally converge slowly: for very small $x$ we will perform many consecutive insertions of mediants. For almost all $x\in[0,1]$ we find 
\[
\lim_{n\to \infty} \frac{1}{n}\log (q_{n,y} )=0.
\]

\subsubsection{Lehner Continued Fractions - \emph{Follow the Edge with the Longer Interval!}}
Do the reverse of the nearest integer continued fraction: follow the unique arrow for which $x_{n+1}>1/2$. This is equivalent to ``pick the substitutions where the word length grows the slowest" or ``construct the first-return map on whichever interval is longer." The effect is to always perform insertions whenever possible, but never singularizations. The corresponding continued fraction representation of $x$ is the unique such expansion for which all partial quotients are one or two. To be more precise, they will be the Lehner continued fraction introduced in \cite{L} and studied in for example \cite{DK}.
The slow approximating sequence includes all possible mediants within this scheme:
\[q_0, 2q_0, \ldots, a_1 q_0, q_2, 2q_2, \ldots, a_3 q_2, q_4, \ldots\]
Note that since the $q_n$ of all previous subsections necessarily appear as a subsequence of the $q_n$ here. Since the previous subsection for almost every $x$ has sub-exponential growth, so do these. For almost every $x \in [0,1]$ we find
\[
\lim_{n\to \infty} \frac{1}{n}\log (q_{n,y} )=0.
\]
\bibliographystyle{alpha}
\bibliography{bibliography}

\end{document}